\documentclass[12pt,reqno]{amsart}
\usepackage{amsmath,amssymb,amsthm,pinlabel}
\usepackage{geometry}
\usepackage[colorlinks=true,
            linkcolor=blue,
            urlcolor=cyan,
            citecolor=magenta]{hyperref}

\usepackage{pinlabel}
\usepackage[full]{textcomp}
\usepackage[osf]{newpxtext} 
\usepackage{cabin} 
\usepackage[varqu,varl]{inconsolata} 
\usepackage[bigdelims,vvarbb]{newpxmath} 
\usepackage{tikz}
\usetikzlibrary{matrix,arrows,decorations.pathmorphing}

\theoremstyle{plain}
\newtheorem{theorem}{Theorem}
\newtheorem{lemma}[theorem]{Lemma}
\newtheorem{proposition}[theorem]{Proposition}
\newtheorem{corollary}[theorem]{Corollary}

\newtheorem*{claim*}{Claim}

\theoremstyle{definition}
\newtheorem{definition}[theorem]{Definition}

\newtheorem{remark}[theorem]{Remark}
\newtheorem{question}[theorem]{Question}

\numberwithin{equation}{section}
\numberwithin{theorem}{section}

\newcommand{\fakeenv}{} 

{ 
 \renewcommand{\fakeenv}{#2} 
 \theoremstyle{plain} 
 \newtheorem*{\fakeenv}{#1~\ref{#2}} 
 \begin{\fakeenv}
}
{
 \end{\fakeenv}
}

\newcommand{\RR}{\mathbb{R}}
\newcommand{\PP}{\mathbb{P}}

\newcommand{\calG}{\mathcal{G}}

\newcommand{\calL}{\mathcal{L}}

\newcommand{\param}%
	{{\mathchoice{\mkern1mu\mbox{\raise2.2pt\hbox{$\centerdot$}}\mkern1mu}%
	{\mkern1mu\mbox{\raise2.2pt\hbox{$\centerdot$}}\mkern1mu}%
	{\mkern1.5mu\centerdot\mkern1.5mu}{\mkern1.5mu\centerdot\mkern1.5mu}}}

\DeclareMathOperator{\Curr}{Curr}
\newcommand{\PCurr}{\PP {\rm Curr}}

\DeclareMathOperator{\Isom}{Isom}

\DeclareMathOperator{\Mod}{Mod}

\DeclareMathOperator{\Teich}{Teich}

\DeclareMathOperator{\ML}{ML}
\DeclareMathOperator{\vol}{vol}
\DeclareMathOperator{\PL}{PL}
\DeclareMathOperator{\Homeo}{Homeo}

\def\strutdepth{\dp\strutbox}
\def \ss{\strut\vadjust{\kern-\strutdepth \sss}}
\def \sss{\vtop to \strutdepth{
\baselineskip\strutdepth\vss\llap{$\diamondsuit\;\;$}\null}}

\def\strutdepth{\dp\strutbox}
\def \sst{\strut\vadjust{\kern-\strutdepth \ssss}}
\def \ssss{\vtop to \strutdepth{
\baselineskip\strutdepth\vss\llap{$\spadesuit\;\;$}\null}}

\def\strutdepth{\dp\strutbox}
\def \ssh{\strut\vadjust{\kern-\strutdepth \sssh}}
\def \sssh{\vtop to \strutdepth{
\baselineskip\strutdepth\vss\llap{$\heartsuit\;\;$}\null}}

\begin{document}



\title[Length functions and applications to dynamics and counting]{Length functions on currents and applications to dynamics and counting}

\author[V. ~Erlandsson ]{Viveka Erlandsson}
\address{School of Mathematics\\
University of Bristol\\
Howard House, Queen's Avenue\\
Bristol BS8 1Sn, UK}
\email{\href{mailto:v.erlandsson@bristol.ac.uk}{v.erlandsson@bristol.ac.uk}}

\author[C.~Uyanik]{Caglar Uyanik}
\address{Department of Mathematics \\
Yale University\\
10 Hillhouse Avenue\\
New Haven, CT 06520, USA}
\email{\href{mailto:caglar.uyanik@yale.edu}{caglar.uyanik@yale.edu}}

%
%

\maketitle
\tableofcontents


\section{Introduction}

Geodesic currents are measure theoretic generalizations of closed curves on hyperbolic surfaces and they play an important role, among many other things, in the study of the geometry of Teichm\"uller space (see, for example, \cite{Bo86, Bo88}).  The set of all closed curves sits naturally as a subset of the space of currents, and various fundamental notions such as geometric intersection number and length of curves extend to this more general setting of currents.

The aim of this (mostly expository) article is twofold. We first explore a variety of length functions on the space of currents, and we survey recent work regarding applications of length functions to counting problems. Secondly, we use length functions to provide a proof of a folklore theorem which states that pseudo-Anosov homeomorphisms of closed hyperbolic surfaces act on the space of projective geodesic currents with uniform north-south dynamics, see Theorem \ref{NSD}.

More precisely, let $S$ be a closed, orientable, finite type surface of genus $g\geq2$ and denote the space of geodesic currents on $S$ by $\Curr(S)$. By a \emph{length function} on $\Curr(S)$ we mean a function that is homogeneous and positive (see Definition \ref{lengthfunction}). There are many ways to define a length of a closed curve on $S$: a Riemannian metric on $S$ naturally induces a notion of length, a generating set of $\pi_1(S)$ gives the notion of word length of a curve,  and given a fixed (filling) curve $\gamma$ one can consider a combinatorial length given by the curves intersection number with $\gamma$. We will see that all these notions of length give rise to \emph{continuous} length functions on $\Curr(S)$. 

As a first example, in section \ref{intnumber}, we will see that Bonahon's intersection form, which is an extension of the geometric intersection number of curves to currents, induces a continuous length function on $\Curr(S)$. We also use this intersection form to prove the uniform north-south dynamics result mentioned above (see section \ref{NSDsection}). 

In section \ref{length} we explore other notions of length of curves that have continuous extensions to length functions on $\Curr(S)$. In section \ref{section:Liouville} we use \emph{Liouville currents} to extend the length of curves given by any (possibly singular) non-positively curved Riemannian metric on $S$ to a continuous length function on $\Curr(S)$, as well as the word length with respect to so-called simple generating sets of $\pi_1(S)$ (see Theorem \ref{nonpositive} and Corollary \ref{vivcor}). Next, we explore length functions with respect to any Riemannian metric on $S$ (respectively the word length with respect to an arbitrary generating set) and explain why the corresponding \emph{stable lengths} extend to continuous length functions on $\Curr(S)$, see Theorem \ref{stable}. 

In section \ref{counting}, we apply the results of section \ref{length} to problems regarding counting curves on surfaces. Mirzakhani \cite{Mir1, Mir2} proved that the asymptotic growth rate of the number of curves of bounded hyperbolic length, in each mapping class group orbit, is polynomial in the length (see Theorem \ref{Maryam} for the precise statement). We explain how to use continuous length functions on $\Curr(S)$ to generalize her result to other notions of length of curves, and show that the same asymptotic behavior holds for all lengths discussed above (see Theorem \ref{juanviveka} and Corollary \ref{EPStheorem}). These results appeared in \cite{ES, viv, EPS} and here we attempt to give a clear outline of the logic behind these proofs. 

\newpage

\noindent \textbf{Acknowledgments:} We are grateful to David Constantine, Spencer Dowdall, Ilya Kapovich, Chris Leininger, Kasra Rafi, Juan Souto and Weixu Su for interesting conversations throughout this project. We thank Dave Futer for asking a question that led to Corollary \ref{futer}. The first author also thanks the School of Mathematics at Fudan University and Vanderbilt University for their hospitality, and gratefully acknowledges support from the NSF grant DMS-1500180 (A. Olshanskii and M. Sapir). Finally, we thank the anonymous referee for helpful comments and corrections. 



\section{Background}

\subsection{Curves on surfaces}

Throughout this article, we let $S$ be a closed, orientable, finite type surface of genus $g\geq 2$. By a \emph{curve} $\gamma$ on $S$ we mean a (free) homotopy class of an immersed, essential, closed curve. That is, the homotopy class of the image of an immersion of the unit circle $S^1\to S$, where the image is not homotopic to a point. We say the curve is \emph{simple} if the immersion is homotopic to an embedding. We identify a curve with its corresponding conjugacy class, denoted $[\gamma]$, in the fundamental group $\pi_1(S)$. Furthermore, we assume curves to be primitive, that is $\gamma\neq\eta^k$ for any $k>1$ and $\eta\in\pi_1(S)$. By a \emph{multicurve} we mean a union of finitely many weighted curves, that is 
$$\bigsqcup_{i=1}^{n} a_i\gamma_i$$
where $a_i>0$ and $\gamma_i$ is a curve for each $i$. We say the multicurve is \emph{integral} if $a_i\in\mathbb{Z}$ for all $i$, and that it is \emph{simple} if the curves $\gamma_i$ are simple and pairwise disjoint.


\subsection{Teichm\"{u}ller space and the mapping class group}

A \emph{hyperbolic structure} on a surface $S$ is a collection of charts $\{(U_i,\psi_i)\}$ such that 
\begin{enumerate}
\item $\{U_{i}\}$ is an open cover of $S$,
\item the map $\psi_i:U_i\to\mathbb{H}^{2}$ is an orientation preserving homeomorphism onto its image for each $i$,
\item For each $i,j$ such that $U_i\cap U_j\neq\emptyset$ the restriction of $\psi_j\circ\psi_i^{-1}$ to each component of $U_i\cap U_j$ is an element of $\Isom^{+}(\mathbb{H}^2)$. 
\end{enumerate}

The surface $S$ together with a hyperbolic structure is called a \emph{hyperbolic surface}. Cartan--Hadamard theorem asserts that a closed hyperbolic surface is isometrically diffeomorphic to $\mathbb{H}^2/\Gamma$ where $\Gamma$ is a torsion free discrete subgroup of $\Isom^{+}(\mathbb{H}^{2})$. 

A \emph{marked hyperbolic surface} is a pair $(X,f)$ where 
\begin{enumerate}
\item $X = \mathbb{H}^{2}/\Gamma$ is a hyperbolic surface, and
\item $f : S\to X$ is an orientation-preserving homeomorphism.
\end{enumerate}

Given a marked hyperbolic surface $(X,f)$, we can pull back the hyperbolic structure on $X$ by $f$ to one on 
$S$. Conversely, given a hyperbolic structure on $S$, the identity map $id: S\to S$ makes $(S,id)$ into a marked hyperbolic surface.

The \emph{Teichm\"uller space} of $S$ is the set $\Teich(S)=\{(X,f)\}/\sim$ of equivalence classes of marked hyperbolic surfaces, where two hyperbolic surfaces $(X,f)$ and $(Y,g)$ are \emph{equivalent} if $g\circ f^{-1}$ is homotopic to an isometry from $X$ to $Y$.

The \emph{mapping class group} $\Mod(S)$ of $S$ is the group of isotopy classes of orientation-preserving homeomorphisms of $S$; in other words,
\[
\Mod(S)=\Homeo^{+}(S)/\Homeo_{0}(S)
\]
where $\Homeo_0(S)$ is the connected component of the identity in the orientation preserving homeomorphism group $\Homeo^+(S)$. 

The mapping class group $\Mod(S)$ acts on $\Teich(S)$ naturally by precomposing the marking map, i.e. for $\varphi\in\Mod(S)$, and $[(X,f)]\in\Teich(S)$ choose a lift $\Phi\in\Homeo^{+}(S)$ of $\varphi$ and define
\[
\varphi[(X,f)]=[(X, f\circ\Phi^{-1})].
\]


\subsection{Measured Laminations}\label{laminations}
A \emph{geodesic lamination} on a hyperbolic surface $S$ is a closed subset $\calL$ of $S$ that is a union of simple, pairwise disjoint, complete geodesics on $S$. The geodesics in $\calL$ are called the \emph{leaves} of the lamination. A \emph{transverse measure} $\lambda$ on $\calL$ is an assignment of a locally finite Borel (Radon) measure $\lambda_{|k}$ on each arc $k$ transverse to $\calL$ so that 
\begin{enumerate}
\item If $k'$ is a subarc of an arc $k$, then $\lambda_{|k'}$ is the restriction to $k'$ of $\lambda_{|k}$;
\item Transverse arcs which are transversely isotopic have the same measure. 
\end{enumerate}

A \emph{measured lamination} is a pair $(\calL,\lambda)$ where $\calL$ is a geodesic lamination and $\lambda$ is a transverse measure. In what follows, we will suppress  $\calL$ and write $\lambda$ for brevity. 
The set of measured laminations on $S$ is denoted by $\ML(S)$, and endowed with the weak-* topology: a sequence $\lambda_n\in\ML(S)$ converges to $\lambda\in\ML(S)$ if and only if 
\[
\int_{k} fd\lambda_n\longrightarrow\int_{k} fd\lambda
\]
for any compactly supported continuous function $f$ defined on a generic transverse arc $k$ on $S$.

An easy example of a measured geodesic lamination is given by a simple curve $\gamma$ on $S$, together with the transverse measure $\lambda_{\gamma}$: for each transverse arc $k$ the transverse measure is the Dirac measure $\lambda_{\gamma|k}$ which counts the number of intersections with $\gamma$, i.e. for any Borel subset $B$ of $k$, we have $\lambda_{\gamma|k}(B)=|B\cap\gamma|$. 
The set of measured geodesic laminations coming from \emph{weighted simple curves} $\lambda_{\gamma}$ 
is dense in $\ML(S)$, see \cite{PH}. We denote the subset of $\ML(S)$ coming from simple integral multicurves by $\ML_{\mathbb{Z}}(S)$. 

Note that $\mathbb{R}^{+}$  acts naturally on $\ML(S)$ by scaling the transverse measure. The space of \emph{projective measured laminations} is defined as the quotient 
\[
\mathbb{P}\ML(S)=\ML(S)/\mathbb{R}_{+}
\]
and the equivalence class of a measured geodesic lamination in $\PP\ML(S)$ is denoted by $[\lambda]$. We endow $\PP\ML(S)$ with the quotient topology.

The space $\ML(S)$ of measured laminations is homeomorphic to $\mathbb{R}^{6g-6}$ and has a $\Mod(S)$-invariant piecewise linear manifold structure (see, for example, \cite{PH}). This piecewise linear structure is given by train track coordinates. We refer the reader to \cite{ThurstonNotes} and \cite{PH} for a detailed discussion of train tracks and only recall the relevant notions for our purposes. 

Let $\tau$ be a smoothly embedded $1$-complex in $S$, i.e. an embedded complex whose edges are smoothly embedded arcs with well-defined tangent lines at the end-points. A \emph{complementary region} of $\tau$ is the metric completion of a connected component of $S\setminus\tau$. We say $\tau$ is a \emph{train track} on $S$ if in addition it satisfies the following properties:
\begin{enumerate}
\item at each vertex, called \emph{switch}, the tangent lines to all adjacent edges agree
\item at each vertex the set of adjacent edges can be divided into two sets according to the direction of the tangent line; we require each of these sets to be non-empty at every vertex
\item doubling each complementary region gives a surface with singular points having negative Euler characteristic $\chi=2-2g-p$ (where $g$ and $p$ represent the genus and the number of singular points, respectively).  
\end{enumerate}
A train track is called \emph{maximal} if the complementary regions to $\tau$ are all triangles.

A simple closed curve, or more generally a measured lamination $(\calL,\lambda)$, is \emph{carried} by $\tau$ if there is a smooth map $g:S\to S$ such that 
\begin{enumerate}
\item $g:S\to S$ is isotopic to the identity,
\item the restriction of $g$ to $\calL$ is an immersion,\item$g(\calL)\subset \tau$.
\end{enumerate}

There is a finite collection of train tracks $\mathcal{T}=\{\tau_1, \tau_2, \ldots, \tau_n\}$ such that for all $\lambda\in\ML(S)$ there is a $\tau_i\in\mathcal{T}$ which $\lambda$ is carried by. The set of measured laminations carried by a train track is full dimensional if and only if the train track is maximal. Now, each maximal train track $\tau$ determines a cone $C(\tau)$ in $\mathbb{R}^{6g-6}$, given by the solutions to the so called \emph{switch equations},
and we have a homeomorphism between all laminations carried by $\tau$ and $C(\tau)$. Moreover, the set of integer points in $C(\tau)$ is in one to one correspondence with the simple integral multicurves (i.e. elements of $\ML_{\mathbb{Z}}(S)$) carried by $\tau$.


\subsection{Geodesic Currents}

Consider a hyperbolic metric $\rho$ on $S$ and let $\tilde{S}$ be the universal cover equipped with the pullback metric $\tilde{\rho}$. Let $\mathcal{G}(\tilde{\rho})$ denote the set of complete geodesics in $\tilde{S}$. Let $S^1_{\infty}$ denote the boundary at infinity of $\tilde{S}$. Note that since $\tilde{\rho}$ is hyperbolic, $\tilde{S}$ is isometric to $\mathbb{H}^2$ and its boundary is homeomorphic to the unit circle $S^1$. Each geodesic is uniquely determined by its pair of endpoints on $S^1_{\infty}$. Hence we can identify the set of geodesics with the \emph{double boundary} 
$$\mathcal{G}(\tilde{S})= \left(S^1_{\infty}\times S^1_{\infty}\right)\setminus\Delta)/(x,y)\sim(y,x)$$
where $\Delta$ denotes the diagonal. That is, $\mathcal{G}(\tilde{S})$ consists of unordered pair of distinct boundary points, and we refer to it as the \emph{space of geodesics of $\tilde{S}$}. Note that $\mathcal{G}(\tilde{S})$ is independent of the metric $\rho$. Indeed, if $\rho'$ is another geodesic metric on $S$, then the universal cover $\tilde{S}$ equipped with the pullback metric $\tilde{\rho}'$ is quasi-isometric to $\mathbb{H}^2$ and this quasi-isometry extends to a homeomorphism of the boundaries at infinity (see \cite{AL} for the details). Hence $\mathcal{G}(\tilde{S})$ is well-defined without a reference to a metric.  
 
The fundamental group $\pi_1(S)$ acts naturally on $\tilde{S}$ by deck transformations, and this action extends continuously to $S^1_{\infty}$ and $\mathcal{G}(\tilde{S})$. For any (geodesic) metric $\rho$ the map
$$\partial_{\rho}: \mathcal{G}(\tilde{\rho})\to \mathcal{G}(\tilde{S})$$
that maps each geodesic to its pair of endpoints is continuous, surjective and $\pi_1(S)$-invariant, and a homeomorphism when $\rho$ is negatively curved. 

A \emph{geodesic current} on $S$ is a locally finite Borel measure on $\mathcal{G}(\tilde{S})$ which is invariant under the action of $\pi_1(S)$. We denote the set of all geodesic currents on $S$ by $\Curr(S)$ and endow it with the weak-* topology: A sequence $\mu_{n}\in\Curr(S)$ of currents converge to $\mu\in\Curr(S)$ if and only if 
\[
\int fd\mu_{n}\longrightarrow \int fd\mu
\]
for all continuous, compactly supported functions $f:\calG(\tilde{S})\to\RR$.  

As a first example of a geodesic current, consider the preimage under the covering map in $\tilde{S}$ of any closed curve $\gamma$ on $S$, which is a collection of complete geodesics in $\tilde{S}$. This defines a discrete subset of $\mathcal{G}(\tilde{S})$ which is invariant under the action of $\pi_{1}(S)$. The Dirac (counting) measure associated to this set on $\mathcal{G}(\tilde{S})$ gives a geodesic current on $S$.

The map from the set of curves on $S$ to $\Curr(S)$ that sends each curve to its corresponding geodesic current, as above, is injective. Hence, we view the set of curves on $S$ as a subset of $\Curr(S)$. In fact, Bonahon showed that the set of all \emph{weighted curves} is dense in $\Curr(S)$ \cite{Bo86}. We identify a curve $\gamma$ with the current it defines, and by abuse of notation we denote both by $\gamma$.

Another important subset of geodesic currents is given by measured laminations. Let $(\calL,\lambda)$ be a measured lamination and consider its preimage $\tilde{\calL}$ in $\tilde{S}$ which is a collection of pairwise disjoint complete geodesics. The lift $\tilde{\calL}$ is a discrete subset of $\calG(\tilde{S})$ which is $\pi_1(S)$-invariant. Hence the associated Dirac measure on $\calL$ defines a geodesic current on $S$. Moreover this measure agrees with the transverse measure $\lambda$, see \cite{AL} for the details. Hence we view $\ML(S)$ as a subset of $\Curr(S)$ as well.  

A current $\nu\in\Curr(S)$ is called \emph{filling} if every complete geodesic in $\tilde{S}$ transversely intersects a geodesic in the support of $\nu$ in $\calG(\tilde{S})$. Note that this definition agrees with the classical notion of \emph{filling curves}: a curve $\gamma$ defines a filling current if and only if $\gamma$ is filling as a curve, i.e. $S\setminus\gamma$ is a union of topological disks. 


\subsection{Nielsen--Thurston classification}

Thurston defined a $\Mod(S)$-equivariant compactification of the Teichm\"uller space by the space of projective measured laminations and using the action of $\Mod(S)$ on $\overline{\Teich(S)}=\Teich(S)\cup\PP\ML(S)$ showed:

\begin{theorem}[Nielsen-Thurston classification]\cite{Th, FLP} Each $\varphi\in \Mod(S)$ is either periodic, reducible or pseudo-Anosov. Furthermore, pseudo-Anosov mapping classes are neither periodic nor reducible. 
\end{theorem} 

Here $\varphi\in \Mod(S)$ is called \emph{periodic} if there exist a $k\ge0$ such that $\varphi^{k}$ is the identity. The map $\varphi$ is called \emph{reducible} if there is a collection $\mathcal{C}$ of disjoint simple curves on $S$ and a representative $\varphi'$ of $\varphi$ such that $\varphi'(\mathcal{C})$ is isotopic to $\mathcal{C}$. Finally, $\varphi\in \Mod(S)$ is called \emph{pseudo-Anosov} if there exists a filling pair of transverse, measured laminations $(\calL^{+},\lambda_{+})$ and $(\calL^{-},\lambda_{-})$, a number $\alpha>1$ called the \emph{stretch factor}, and a representative homeomorphism $\varphi'$ of $\varphi$ such that 
\[
\varphi'(\calL^{+},\lambda_{+})=(\calL^{+},\alpha\lambda_{+})
\]
and 
\[
\varphi'(\calL^{-},\lambda_{-})=(\calL^{-},\frac{1}{\alpha}\lambda_{-}).
\]
The measured laminations $(\calL^{+},\lambda_{+})$ and $(\calL^{-},\lambda_{-})$ are called the \emph{stable lamination} and the \emph{unstable lamination} respectively. We will suppress the $\calL$ and write $\lambda_{+}$ and $\lambda_{-}$  respectively.


\subsection{Length functions and the intersection number}\label{intnumber}

\begin{definition}\label{lengthfunction} A \emph{length function} on the space of geodesic currents is a map $\ell :\Curr(S)\to\RR$ which is homogeneous and positive, i.e.
\[
\ell(a\mu)=a\ell(\mu)
\]
for any $a>0$ and $\mu\in\Curr(S)$, $\ell(\mu)\ge0$ for all $\mu\in\Curr(S)$ and  $\ell(\mu)=0$ iff $\mu=0$. 

\end{definition}

We say that a map $\ell$ on the set of curves on $S$ is a length function if it is a positive function, i.e. $\ell(\gamma)>0$ for all curves $\gamma$ on $S$. Note that this agrees with the definition above, when viewing the set of curves as a subset of the space of geodesic currents.

Given two curves $\gamma, \eta$ on $S$, their \emph{(geometric) intersection number}, denoted $i(\gamma, \eta)$, is defined as the minimum number of transverse intersections between transverse representatives of the homotopy classes of $\gamma$ and $\eta$.  That is
$$i(\gamma, \eta) = \min\left\{\vert\gamma'\pitchfork\eta'\vert\,\vert\, \gamma'\sim\gamma, \eta'\sim\eta \right\}$$
where $\sim$ denotes homotopic. 

We note that $i(\gamma, \gamma)=0$ if and only if $\gamma$ is a simple curve. Moreover, an equivalent description of the intersection number of two distinct curves $\gamma$ and $\eta$ is the following. Let $\rho$ be a hyperbolic metric on $S$ and $\tilde{S}$ be the universal cover equipped with the pullback metric $\tilde{\rho}$.  Let $\tilde{\gamma}$ be a geodesic representative of a lift of $\gamma$ to $\tilde{S}$. Consider the set of lifts of $\eta$ and take their geodesic representatives.  Let $x$ be a point on $\tilde{\gamma}$ that does not lie on a geodesic representative of any lift of $\eta$,  and consider the bounded segment $\delta_{\gamma}$ of $\tilde{\gamma}$ between $x$ and $\gamma(x)$. Then the intersection number $i(\gamma, \eta)$ is exactly the same as the number of the lifts of $\eta$ that intersect (necessarily transversely) $\delta_{\gamma}$. This description of the intersection number will be helpful below. 

Viewing the set of curves as a subspace of the space of geodesic currents, it is natural to ask if the intersection number extends, in a nice way, to $\Curr(S)$. Indeed, Bonahon \cite{Bo88} showed that there is a unique continuous extension of the intersection number to the space of geodesic currents:

\begin{theorem}\cite[Proposition 4.5]{Bo88}
There is a unique continuous, symmetric, bilinear form 
$$i(\cdot,\cdot): \Curr(S)\times\Curr(S)\to\mathbb{R}_{\geq0}$$
such that 
$i(\gamma, \eta)$ agrees with the geometric intersection number whenever $\gamma, \eta$ are curves on $S$. 
\end{theorem}

Here we give the definition of this intersection form and explain how it induces length functions on $\Curr(S)$. For the definition we follow the exposition presented in \cite{AL} and refer to that paper for more details. Let $\mathcal{G}^{\pitchfork}(\tilde{S})\subset\mathcal{G}(\tilde{S})\times\mathcal{G}(\tilde{S})$ be the subset defined by
$$\mathcal{G}^{\pitchfork}(\tilde{S}) = \left\{\left(\{x,y\}, \{z,w\}\right)\in\mathcal{G}(\tilde{S})\times\mathcal{G}(\tilde{S})\setminus\Delta\,\vert\{x,y\}, \{c,d\}\mbox{ link}\right\}$$
where $\Delta$ represents the diagonal and we say that $\{x,y\}$ and $\{z,w\}$ \emph{link} if $x$ and $y$ belong to different components of $S^1_{\infty}\setminus\{z,w\}$. Equivalently, $\mathcal{G}^{\pitchfork}(\tilde{S})$ consists of pairs of geodesics in $\tilde{S}$ that intersect transversely. The action of $\pi_1(S)$ on $\tilde{S}$ induces a free and properly discontinuous action on $\mathcal{G}^{\pitchfork}(\tilde{S})$ and hence the quotient map 
$$\mathcal{G}^{\pitchfork}(\tilde{S})\to\mathcal{G}^{\pitchfork}(\tilde{S})/\pi_1(S)$$
is a covering map. We define
$$\mathcal{G}^{\pitchfork}(S)=\mathcal{G}^{\pitchfork}(\tilde{S})/\pi_1(S).$$
Now, let $\mu, \nu\in\Curr(S)$. Then $\mu\times\nu$ is a product measure on $\mathcal{G}(\tilde{S})\times\mathcal{G}(\tilde{S})$ and hence on $\mathcal{G}^{\pitchfork}(\tilde{S})$. This descends to a measure on $\mathcal{G}^{\pitchfork}(S)$ by locally pushing forward $\mu\times\nu$ through the covering map, and the intersection of $\mu$ and $\nu$ is defined as the $\mu\times\nu$-mass on $\mathcal{G}(S)$. That is  
$$i(\mu, \nu)=\int_{\mathcal{G}^{\pitchfork}(S)} d\mu\times d\nu.$$

Let $\gamma$ be a curve and identify it with the current it defines. Let $\mu\in\Curr(S)$. Then $i(\gamma, \mu)$ can be defined as follows. As above, choose a hyperbolic metric $\rho$ on $S$ and consider the universal cover $\tilde{S}$ equipped with the pullback metric $\tilde{\rho}$. Take a lift of $\gamma$ and let $\tilde{\gamma}$ be its geodesic representative. Let $x$ be a point on $\tilde{\gamma}$ and consider the geodesic segment $\eta_{\gamma}$ from $x$ to $\gamma(x)$. Let $\mathcal{G}^{\pitchfork}(\eta_{\gamma})$ denote the set of geodesics that transversely intersect  $\eta_{\gamma}$ and $\partial_{\rho}\mathcal{G}^{\pitchfork}(\eta_{\gamma})$ the subset of $\mathcal{G}(\tilde{S})$ obtained by identifying each geodesic in $\mathcal{G}^{\pitchfork}(\eta_{\gamma})$ with its pair of endpoints on $S^1_{\infty}$. Then 
$$i(\gamma,\mu)=\mu\left(\partial_{\rho}\mathcal{G}^{\pitchfork}(\eta_{\gamma})\right),$$
see Figure \ref{intersectionfigure}. In particular, we see that when $\mu$ is also (the current associated to) a curve on $S$, then the intersection form agrees with the geometric intersection number of curves on $S$. 

\begin{figure}[h!]\label{intersectionfigure}
\labellist
\small\hair 2pt
 \pinlabel {\tiny $x$} [ ] at 245 590
 \pinlabel {$\bullet$} [ ] at 253 599
 \pinlabel {$\bullet$} [ ] at 302 605
 \pinlabel {\tiny$\tilde{\gamma}x$} [ ] at 312 596
 \pinlabel {\textbf{I}} [ ] at 280 490
 \pinlabel {\textbf{J}} [ ] at 270 738
\endlabellist
\centering
\includegraphics[scale=0.65]{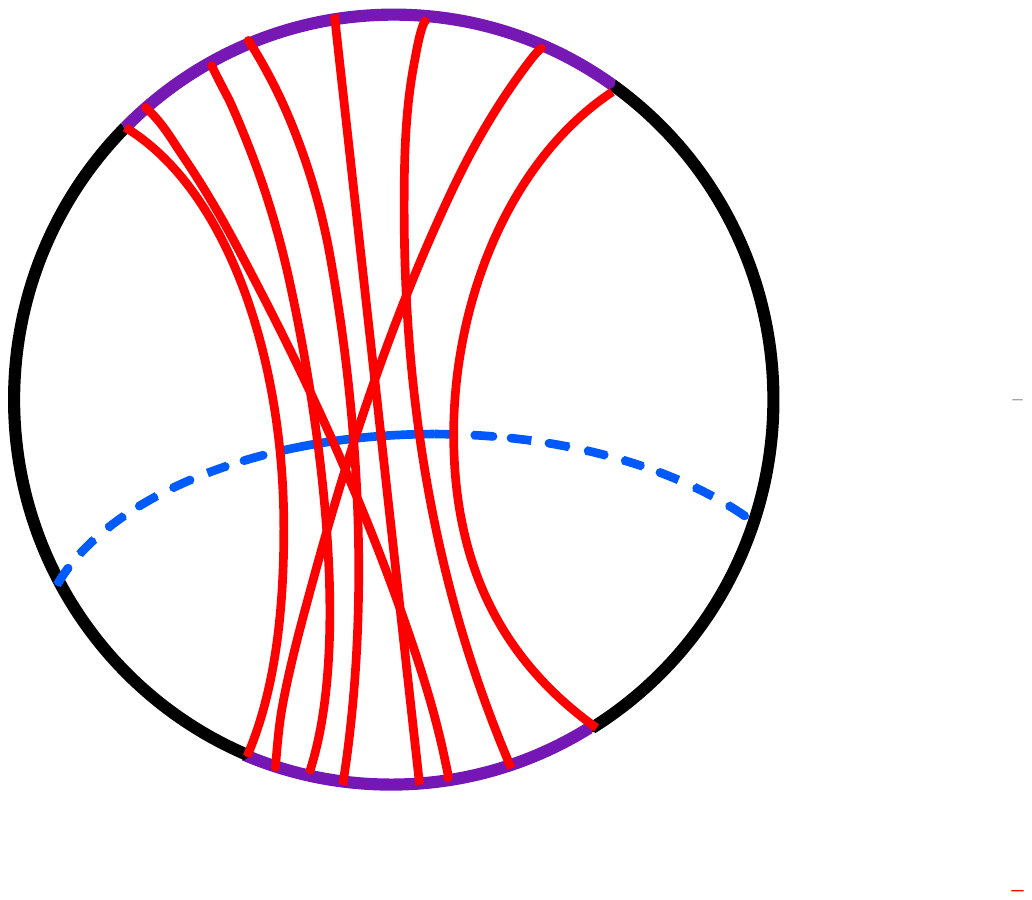}
\caption{Intersection of a curve $\gamma$ with a current $\mu$: $i(\gamma,\mu)=\mu(I\times J)$. Here $\eta_{\gamma}$ is represented by the solid blue segment and $\partial_{\rho}\mathcal{G}^{\pitchfork}(\eta_{\gamma})=I\times J$.}
\label{intersectionfigure}
\end{figure}

We record some useful facts about the intersection form:
\begin{enumerate}
\item If $\nu\in\Curr(S)$ is filling, then $i(\nu,\mu)\neq0$ for all $\mu\in\Curr(S)\setminus\{0\}$\label{step nonzero}.
\item The intersection form is invariant under $\Mod(S)$. That is, if $g\in\Mod(S)$ then $i(\mu, \nu) = i(g(\mu), g(\nu))$ for any $\mu, \nu\in\Curr(S)$.
\item $i(\mu, \mu)=0$ if and only if $\mu\in\ML(S)$.
\item If $\nu\in\Curr(S)$ is filling and $L>0$, then the set\label{step compact}
$$\{\mu\in\Curr(S)\,\vert\,i(\mu,\nu)\leq L\}$$
is a compact set.
\end{enumerate}
The first two statements follow from the definition of the intersection form, while the last two are results by Bonahon, see \cite{Bo86}.
 
We define the space of \emph{projective geodesic currents} to be 
$$\PCurr(S) = \left(\Curr(S)\setminus\{0\}\right)/\mathbb{R}_{+}.$$
It follows from (\ref{step compact}) above that $\PCurr(S)$ is a compact space. 

Next we show how to obtain continuous length functions on $\Curr(S)$ from the intersection form. Fix a filling current $\nu\in\Curr(S)$. Define
$$\ell_{\nu}(\mu)=i(\nu,\mu)$$
for all $\mu\in\Curr(S)$. By the linearity and continuity of the intersection form, $\ell_{\nu}$ is continuous and homogenous on $\Curr(S)$. Furthermore, since $\nu$ is filling, it follows from (\ref{step nonzero}) above that $\ell_{\nu}$ is positive.  Hence the function above defines a continuous length function 
$$\ell_{\nu}:\Curr(S)\to\mathbb{R}.$$
Moreover, this is the unique continuous extension of the length function on the set of curves defined by
$$\ell_{\nu}(\gamma)=i(\nu, \gamma)$$
for all curves $\gamma$ on $S$. In section \ref{length} we will see that many other notions of lengths of curves have unique continuous extensions to length functions on $\Curr(S)$. 

We end this section by noting that the intersection form can also be defined for geodesic currents on surfaces with boundary, and we refer the reader to \cite{DLR} 
for the definitions. For simplicity of the exposition we assume throughout that $S$ is a closed surface although the results presented here have generalizations that also hold for the case of compact surfaces. 



\section{Length functions on space of currents}\label{length}


In section \ref{intnumber} we saw that the geometric intersection number on the set of curves extend continuously to a bilinear form on pairs of currents and hence, fixing a filling curve (or current) $\nu$, the length function 
$$\ell_{\nu}(\gamma)=i(\nu,\gamma)$$
on the set of curves, extends \emph{continuously} to a length function on currents, defined by
$$\ell_{\nu}(\mu)=i(\nu, \mu)$$
for all $\mu\in\Curr(S)$. There are many ways to define the length of a curve on a surface and it is natural to ask which other notions of length extends continuously to the space of currents. More concretely, let $\rho$ be a (possibly singular) Riemannian metric on $S$. Then $\rho$ naturally induces a length function $\ell_{\rho}(\cdot)$ on the set of curves where the length of a curve $\gamma$ is defined to be the $\rho$-length of a shortest representative of $\gamma$. In the case when $\rho$ is a negatively curved metric, this is the length of the unique geodesic representative in the homotopy class of $\gamma$. Another natural length function on the set of curves is given by first identifying a curve on $S$ with a conjugacy class in the fundamental group $\pi_1(S)$ and, for a fixed a generating set of $\pi_1(S)$, defining the length of a curve to be the minimal number of generators needed to represent the corresponding conjugacy class. In general, given a geodesic metric space $(X, d)$ on which $\pi_1(S)$ acts discretely and cocompactly by isometries, one can ask whether the translation length of a curve $\gamma$
\begin{equation}\label{translation length}
\ell_X(\gamma)=\inf_{x\in X}d(x,\gamma(x))
\end{equation}
extends continuously to a length function on the space of currents. Note that when $X$ is the universal cover of $S$ equipped with a Riemannian metric, or $X$ is the Cayley graph with respect to a generating set of $\pi_1(S)$, this length agrees with the notions described above. 

We will see that in many cases such a continuous extension exists. In particular, in section \ref{section:Liouville} below, we explain why it exists for any (possibly singular) non-positively curved Riemannian metric on $S$, through the use of \emph{Liouville currents} and their relation to the intersection form on $\Curr(S)$. Similar arguments show that the word length with respect to certain (well-chosen) generating sets extends continuously to a length function on the space of currents. 

Alas, such a continuous extension does not always exist. However, as we will see in section \ref{section:stable}, for any length function $\ell_X$ on curves as above, the \emph{stable length} function defined by
\[
sl_X(\gamma)=\lim_{n\to\infty}\frac{1}{n}\ell_X(\gamma^n)
\]
always extends continuously to a length function on $\Curr(S)$.  

Finally, in section \ref{section:glorious}, we will see that the two approaches of defining an extension (using intersection with a special current, and considering the stable length) are intimately related.


\subsection{Length of currents through Liouville currents}\label{section:Liouville}

In this section we explain how the length function of curves with respect to any (possibly singular) non-positively curved Riemannian metric on $S$ can be extended continuously to a length function on the space of currents on $S$. 

First, fix a hyperbolic metric $\rho$ on $S$. The hyperbolic length of a homotopy class of a closed curve $\gamma$ is defined as the length of the $\rho$-geodesic representative, and denoted by $\ell_{\rho}(\gamma)$. There exists a current associated with $\rho$, called its \emph{Liouville current} and denoted by $L_{\rho}$, whose intersection form with curves on $S$ determines the length function induced by $\rho$, that is:
\begin{equation}\label{e:Liouville}
i(\gamma,L_{\rho})=\ell_{\rho}(\gamma)
\end{equation}
for all curves $\gamma$ on $S$. 

Here we describe two equivalent definitions of the Liouville current and we refer the reader to \cite{AL, Bo88, HerPa, Otal} for more details.  
First we define the \emph{Liouville measure} $L$ on the hyperbolic plane $\mathbb{H}^2$. Let $\mathcal{G}(\mathbb{H}^2)$ denote the set of all bi-infinite geodesics in $\mathbb{H}^2$, which we identify with their endpoints on the unit circle $S^1$. Let $[a,b]$ and $[c,d]$ be two non-empty disjoint intervals on $S^1$. Define 
\begin{equation}\label{cross}
L\left([a,b]\times[c,d]\right)=\left\lvert\log\left\lvert\frac{(a-c)(b-d)}{(a-d)(b-c)}\right\rvert\right\rvert
\end{equation}
whenever $a, b, c, d$ are four distinct points, and set $L\left([a,b]\times[c,d]\right)=0$ if one of the intervals is a singleton. The map $L$ extends uniquely to a Radon measure on $\mathcal{G}(\mathbb{H}^2)$ (see \cite{Bo88}) and is invariant under the action of $\pi_1(S)$ since the right-hand side in \eqref{cross} is invariant under this action. In the disk model of $\mathbb{H}^2$ we have, using local coordinates $(e^{i\alpha},e^{i\beta})$,
$$L_{\rho}=\frac{d\alpha d\beta}{\vert e^{i\alpha}-e^{i\beta}\vert^2}$$
where $d\alpha d\beta$ is the Lebesgue measure defined by the Euclidean metric on $S^1$, and in particular $L$ is absolutely continuous with respect to the Lebesgue measure (see, for example, \cite{AL}). Now, given a hyperbolic metric $\rho$ on $S$ the universal cover $\tilde{S}$ with the pull-back metric $\tilde{\rho}$ is isometric to $\mathbb{H}^2$ and the boundary $S^1_{\infty}$ is homeomorphic to $S^1$. We define $L_{\rho}$, the Liouville current with respect to $\rho$, to be the pull-back of $L$ through this homeomorphism. 

Alternatively, one can define the Liouville current in the following way. Let $\tilde\eta$ be a $\tilde{\rho}$-geodesic arc in $\tilde{S}$, parametrized at unit-speed by $\tilde{\eta}: (-a, a)\to\tilde{\eta}(t)$. Let $\mathcal{G}(\tilde\eta)$ denote the set of all $\tilde\rho$-geodesics in $\tilde{S}$ that intersect $\tilde\eta$ transversely. Note that each geodesic in $\mathcal{G}(\tilde\eta)$ is uniquely determined by its point of intersection $\tilde{\eta}(t)$ with $\tilde{\eta}$ and its angle of intersection (chosen in an arbitrary but consistent way). This gives rise to a homeomorphism 
$$h_{\eta}:(-a,a)\times(0,\pi)\to \mathcal{G}(\tilde{\eta}).$$
Consider the measure on $(-a,a)\times(0,\pi)$ defined by
$$ds=\frac{1}{2}\sin(\theta)d\theta dt.$$
We push forward this measure through $h_{\eta}$ to obtain a measure on $\mathcal{G}(\tilde{\eta})$. Lastly, we further push the measure forward through the homeomorphism 
$$\partial_{\rho}: \mathcal{{G}}(\tilde{\rho})\to\mathcal{G}(\tilde{S})$$
which maps each geodesic in $\mathcal{G}(\tilde{\eta})$ to its endpoints. The resulting measure is a Radon measure on $\mathcal{G}(\tilde{S})$. Furthermore, since $\pi_1(S)$ acts by isometries on $\tilde{S}$, the measure is invariant under its action. This measure is the Liouville measure $L_{\rho}$ and agrees with the previous definition.  

While the closed formula in the first definition makes it easier to state, the construction involved in the latter makes \eqref{e:Liouville} more natural to see. Indeed, integrating $ds$ over $\mathcal{G}(\tilde{\eta})$ for a unit-speed parametrized geodesic arc $\tilde{\eta}$ gives exactly the length of $\tilde\eta$.  

The existence of Liouville currents for hyperbolic metrics allows us to embed the Teichm\"{u}ller space of $S$ into the space of geodesics currents, as shown by Bonahon \cite{Bo88}. More precisely, let $(X,f)$ be a point in the Teichm\"uller space, and $\ell_{X}$ and $L_{X}$ be the corresponding length function on curves and the Liouville current, respectively. Then, we have: 

\begin{theorem}\cite{Bo88}\label{Bonahon length}
The map 
$$(X,f)\mapsto L_{X}$$
defines an embedding $\Teich(S)\hookrightarrow\Curr(S)$ satisfying
$$i(\gamma,L_{X})=\ell_{X}(\gamma)$$
for all curves $\gamma$ on $S$. 
\end{theorem}

Note that, since the intersection form is continuous and bilinear on $\Curr(S)\times\Curr(S)$, as discussed in section \ref{intnumber}, the hyperbolic length function has a continuous extension to a length function on $\Curr(S)$ by setting $$\ell_{\rho}(\mu)=i(\mu,L_{\rho})$$
for all $\mu\in\Curr(S)$. The positivity of this function follows from the fact that the Liouville current is filling and hence $i(\mu, L_{\rho})=0$ if and only if $\mu$ is the $0$-current. Moreover, this extension is unique due to the following theorem by Otal \cite{Otal}.

\begin{theorem}\cite{Otal}
Suppose $\mu_1, \mu_2\in\Curr(S)$. If $i(\mu_1,\gamma)=i(\mu_2, \gamma)$ for all curves $\gamma$ on $S$, then $\mu_1=\mu_2$. 
\end{theorem}

More generally, let $\rho$ be any metric on $S$, and let $\ell_{\rho}(\gamma)$ denote the length of a shortest representative in the homotopy class of a curve $\gamma$. We say $L_{\rho}$ is a \emph{Liouville current for $\rho$} if equation \eqref{e:Liouville} holds, that is
$$i(\gamma, L_{\rho})=\ell_{\rho}(\gamma)$$ 
for all curves $\gamma$ on $S$. Note that when such a current exists it must be unique and is necessarily a filling current, by the same theorem by Otal.  

As explained above, a Liouville current exists for any hyperbolic metric on $S$. Otal \cite{Otal} showed the existence of a Liouville current for any (variable) negatively curved metric on $S$. By work of Duchin-Leininger-Rafi \cite{DLR} and Bankovic-Leininger \cite{BL} such a current also exists for any non-positively curved Euclidean cone metric on $S$. Finally, Constantine \cite{Constantine} extended these results to any non-positively curved (singular) Riemannian metric, giving the Liouville current associated to any such metric (in fact, also for the larger class of so-called \emph{no conjugate points cone metrics}, see \cite{Constantine} for the definition). We record a consequence of this sequence of results here: 

\begin{theorem}\cite[Proposition 4.4]{Constantine}\label{nonpositive}
Let $\rho$ be any (possibly singular) non-positively curved Riemannian metric on $S$ and let $\ell_{\rho}(\gamma)$ denote the $\rho$-length of a shortest representative in the homotopy class of $\gamma$. Then the length function $\ell_{\rho}$ on the set of curves extends continuously to a length function 
$$\ell_{\rho}: \Curr(S)\to\mathbb{R}.$$
Moreover, this extension is unique.   
\end{theorem}

We note that Liouville currents also exist in other settings. Notably, Martone--Zhang proved the existence of such currents in the context of a large class of representations, including Hitchin and maximal ones, see \cite{MZ} for details. In another direction, Sasaki proved the existence of a bilinear intersection functional on the space of subsets currents on surfaces, and proved Liouville type equalities, \cite{KN, SaS}. We refer interested reader to these papers as they are beyond the scope of this paper. 

From a more algebraic viewpoint, one can consider the word metric on $\pi_1(S)$ with respect to a fixed generating set: We choose a base point $p$ on $S$ and identify the elements of $\pi_1(S)$ with loops based at $p$.  Since this group is finitely generated, we choose a finite, symmetric generating set $G=\{g_1^{\pm1},g_2^{\pm1},\ldots,g_n^{\pm1}\}$. Given a conjugacy class $[\gamma]$ (or, equivalently, a homotopy class of a curve $\gamma$) we define the word length of the conjugacy class $[\gamma]$ with respect to $G$ to be 
$$\ell_{G}([\gamma]) = \min\left\{\vert k_1\vert+\vert k_2\vert+\cdots\vert k_m\vert\,\vert\, g_{i_1}^{k_1}g_{i_2}^{k_2}\cdots g_{i_m}^{k_m}\in [\gamma]\right\}.$$ 

We say a generating set $G$ is \emph{simple} if the loops $g_i$ in $G$ are simple and pairwise disjoint except at the base point $p$ (see Figure \ref{simplegenerator} for an example). Note that there are many such generating sets, including any one vertex triangulation of $S$ or the standard generating set for a genus $g$ surface $\{a_1,b_1,a_2,b_2,\ldots a_g,b_g\}$ with the relation $[a_1,b_1]\cdots[a_g,b_g]=1$.

\begin{figure}[h!]
\labellist
\small\hair 2pt
 \pinlabel {$\alpha$} [ ] at 150 550
 \pinlabel {$\beta$} [ ] at 360 550
 \pinlabel {$\eta$} [ ] at 237 560
 \pinlabel {$\delta$} [ ] at 215 460
 \pinlabel {$\gamma$} [ ] at 280 460
\endlabellist
\centering
\includegraphics[scale=0.50]{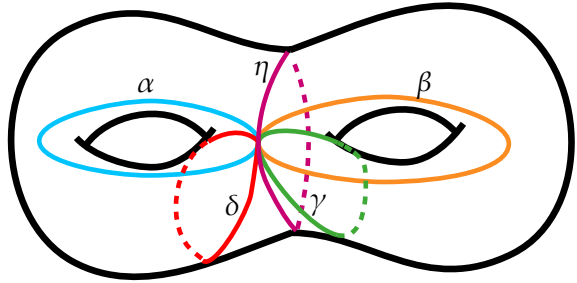}
\caption{A genus $2$ surface $S$ with a simple (non-minimal) generating set $G=\{\alpha^{\pm1}, \beta^{\pm1}, \gamma^{\pm1}, \delta^{\pm1}, \eta^{\pm1}\}$.}
\label{simplegenerator}
\end{figure}


In \cite{viv} it is shown that, given a simple generating set $G$, there exists a collection of curves $\nu=\nu(G)$, depending only on the generating set, such that the word length of a curve is given exactly by its geometric intersection number with this curve: 

\begin{theorem}\cite[Theorem 1.2]{viv}
Let $G$ be a simple generating set for $\pi_1(S)$. Then there exists a collection of curves $\nu=\nu(G)$ on $S$ such that 
$$\ell_{G}(\gamma) = i(\nu, \gamma)$$
for all curves $\gamma$ in $S$. Moreover, $\nu$ is unique with this property. 
\end{theorem} 

By viewing the set of curves as a subset of $\Curr(S)$, if $G$ is a simple generating set for $\pi_1(S)$, then the above results says that there exists a (unique) Liouville current associated to the corresponding word metric. In particular, it follows that the word length extends continuously to the space of currents: 

\begin{corollary}\cite[Corollary 1.3]{viv}\label{vivcor}
Let $G$ be a simple generating set for $\pi_1(S)$. Then the word length with respect to $G$ on the set of curves extends continuously to a length function 
$$\ell_{G}: \Curr(S)\to\mathbb{R}.$$
Moreover, this extension is unique.  
\end{corollary}


\subsection{Stable length of currents}\label{section:stable}

There are many notions of lengths not covered by the Liouville currents explained above. Two such occasions are the length of a curve with respect to a Riemannian metric which attains positive curvature values at places, and the word length with respect to a non-simple generating set. In fact, in these settings such currents do not necessarily exist. For instance, if we consider the word metric with respect to a non-simple generating set then we observe that the length function \emph{cannot} extend continuously to a length function on the space of geodesic currents. To see this, consider the case where $S$ is the once-punctured torus and let $a$, $b$ be the standard generators for the free group $\pi_1(S)$. Then, the word length with respect to the generating set $G=\{a^{\pm1}, b^{\pm1}, a^{\pm2}\}$, does not extend to a continuous homogeneous function on $\Curr(S)$. Indeed, the sequence of currents $\left(\frac{1}{2n}a^{2n}b\right)$ converges to the current $a$ as $n\to\infty$ and hence if such a function $\ell_{G}$ existed, continuity would imply that 
$$\ell_{G}\left(\frac{1}{2n}a^{2n}b\right)\to\ell_{G}(a)=1$$ 
while, on the other hand, homogeneity would imply 
$$\ell_{G}\left(\frac{1}{2n}a^{2n}b\right)=\frac{1}{2n}\ell_G(a^{2n}b) = \frac{n+1}{2n}\to\frac{1}{2}$$
as $n\to\infty$, a contradiction. 

However, as shown in \cite {EPS}, if we consider the \emph{stable length} of curves instead, which we describe below, this length function always extends continuously to the space of geodesic currents. 

Let $X$ be any geodesic metric space on which $\pi_1(S)$ acts discretely and cocompactly by isometries. For a conjugacy class $[\gamma]$ in $\pi_1(S)$ (or, equivalently, a curve $\gamma$ on $S$), define its translation length $\ell_X(\gamma)$ with respect to $X$ as in \eqref{translation length}. Then the \emph{stable length} of $[\gamma]$ is defined to be
\[
sl_X(\gamma)=\lim_{n\to\infty}\frac{1}{n}\ell_X(\gamma^n)=\lim_{n\to\infty}\frac{1}{n}\inf_{x\in X}d(x,\gamma^{n}(x)).
\] 
Again, this definition is independent of the choice of the representative in the conjugacy class. In \cite{EPS} it is shown that, with $X$ as above, this notion of length always extends continuously to $\Curr(S)$:

\begin{theorem}\cite[Theorem 1.5]{EPS}\label{stable}
Let X be a geodesic metric space on which $\pi_1(S)$ acts discretely and cocompactly by isometries. Then the stable length function $sl_X$ on the set of curves extends continuously to a length function 
$$sl_X:\Curr(S)\to\mathbb{R}.$$
Moreover, this extension is unique. 
\end{theorem}

The proof of Theorem \ref{stable} is rather involved, with the main difficulty being how to define the stable length of a current, and we will not explain it here, but we refer the reader to \cite{EPS}. Instead we give some consequences of Theorem \ref{stable}.  
If we equip $S$ with any Riemannian metric and let $X$ be its universal cover $\tilde{S}$ we immediately get the following corollary:

\begin{corollary}\label{riemannian}
Let $\rho$ be any Riemannian metric on $S$. For a  curve $\gamma$, let $\ell_{\rho}(\gamma)$ be the $\rho$-length of a shortest representative. Then the stable length defined by 
$$st_{\rho}(\gamma)=\lim_{n\to\infty}\frac{1}{n}\ell_{\rho}(\gamma^n)$$
has a unique continuous extension to a length function 
$$st_{\rho}: \Curr(S)\to \mathbb{R}_{+}.$$
\end{corollary}

Similarly, if we let $X$ be the Cayley graph with respect to a finite generating set of $\pi_1(S)$ we also have: 

\begin{corollary}\label{word}
Let $G$ be any finite generating set for $\pi_1(S)$. Let $\ell_{G}(\gamma)$ denote the shortest word length of a representative in the conjugacy class of $\gamma$. Then the stable length defined by 
$$st_{G}(\gamma)=\lim_{n\to\infty}\frac{1}{n}\ell_{G}(\gamma^n)$$
has a unique continuous extension to a length function
$$st_{G}: \Curr(S)\to \mathbb{R}_{+}.$$
\end{corollary}

We remark that in \cite{EPS} Theorem \ref{stable} was proved in a more general setting, namely when the surface group is replaced by any torsion free Gromov hyperbolic group $\Gamma$. It is shown that in this setting, the corresponding stable length of a conjugacy class extends to a continuous length function on the \emph{space of (oriented) currents on $\Gamma$}. This space, introduced by Bonahon \cite{Bo91} and denoted $\Curr(\Gamma)$, is defined to be the set of $\Gamma$-invariant Radon measures on the double boundary 
$$\left(\partial\Gamma\times\partial\Gamma\setminus\Delta\right)/\sim$$
where $\partial\Gamma$ is the Gromov boundary of $\Gamma$ and where we identify $(\gamma_1, \gamma_2)$ with $(\gamma_2, \gamma_1)$ (see, for example, \cite{col:KB02}). Since we will not use this more general setting here we refer to \cite{Bo91} and \cite{EPS} for the precise definitions.
 
We also remark that Theorem \ref{stable} was proved by Bonahon \cite{Bo91} in the case when $X$ is "uniquely geodesic at infinity", i.e. any two points on the (Gromov) boundary at infinity of $X$ determine a unique geodesic between them. However, this condition is not satisfied in general for the universal cover of Riemannian metrics, nor for Cayley graphs. 

Finally we note that, in \cite{Bo91}, Bonahon remarks that it should be possible to remove not only the uniquely geodesic hypothesis, which Theorem \ref{stable} proves, but also the cocompact assumption. The proof of Theorem \ref{stable} in \cite{EPS} still requires $\Gamma$ to act cocompactly on $X$ and it is an interesting question whether it is a necessary condition. 

\begin{question}
Does Theorem \ref{stable} still hold for a surface group that acts discretely, but not cocompactly on $X$?
\end{question}

It should be noted that the assumption on the action to be discrete cannot be removed, as shown by Bonahon \cite{Bo91}.


\subsection{Stable length as a generalization of intersection length}\label{section:glorious}

At first glance, extending length of curves to length functions on currents through the intersection length or by considering the stable length might seem like very different approaches. However, as we will observe below, the two notions can be unified: given a filling current $\nu$ one can construct a metric space $(X,d)$ on which $\pi_1(S)$ acts discretely and cocompactly by isometries, and such that 
$$sl_X(\gamma)=i(\nu, \gamma)$$
for all curves $\gamma$ on $S$. The basis for our metric is a semi-distance presented by Glorieux in \cite{Glorious}, described below. 

Fix a hyperbolic metric $\rho$ on $S$ and let $\tilde{S}$ be the universal covering equipped with the pull back metric. We define a metric space $(X,d)$ in the following way. As in section \ref{section:Liouville}, for a geodesic arc $\tilde{\eta}$ let $\mathcal{G}(\tilde{\eta})$ denote the set of geodesics in $\tilde{S}$ that intersect $\tilde{\eta}$ transversely. Let $\partial_{\rho}\mathcal{G}(\tilde{\eta})$ denote the image of $\mathcal{G}(\tilde{\eta})$ under the homeomorphism that maps each geodesic to its pair of endpoints. Let $\nu$ be a filling current in $\Curr(S)$. For two distinct points $x,y\in\tilde{S}$, define
$$d'(x,y)=\nu(\partial_{\rho}\mathcal{G}(\tilde{\eta}))$$
where $\tilde{\eta}$ is the geodesic arc connecting $x$ and $y$. Set $d'(x,x)=0$ for all $x\in\tilde{S}$. 

Note that $d'$ is symmetric, i.e. $d'(x,y)=d'(y,x)$, and $d'(x,y)\geq0$ for all $x,y\in\tilde{S}$ (although $d'$ might not separate points). Furthermore, by definition of the intersection number (see section \ref{intnumber}), if $x$ lies on the axis of an element $\gamma\in\pi_1(S)$ then 
\begin{equation}\label{equality}
d'(x,\gamma(x))=i(\nu,\gamma). 
\end{equation}
Moreover, in \cite{Glorious} it is shown that 
\begin{enumerate}
\item $d'$ satisfies the triangle inequality, i.e. $d'(x,y)\leq d'(x,z) + d'(z,y)$ for all $x,y,x\in\tilde{S}$, 
\item $i(\nu, \gamma)\leq d'(x, \gamma(x))$ for all $x\in\tilde{S}$ and $\gamma\in\pi_1(S)$. \label{inequality}
\end{enumerate}

In particular, $d'$ is a semi-distance. In \cite{Glorious} $d'$ was used to find the critical exponent for geodesic currents, here we use it to construct our desired metric space. Define
$$X=\tilde{S}/\sim$$
where $x\sim y$ if and only if $d'(x,y)=0$, equipped with the metric $d$ induced by $d'$. That is, 
$$d\left([x],[y]\right) = d'(x,y)$$ 
for all $[x],[y]\in X$, where $x$ and $y$ are any representatives of $[x]$ and $[y]$, respectively. 
Using \ref{inequality} above and equation \eqref{equality} we see that the stable length with respect to $X$ agrees with the length function defined by intersection with $\nu$:
$$sl_X(\gamma)=\lim_{n\to\infty}\frac{1}{n}\inf_{x\in X}d(x,\gamma^{n}(x))=\lim_{n\to\infty}\frac{1}{n}i(\nu,\gamma^n) = i(\nu, \gamma)$$
for any conjugacy class $[\gamma]$ in $\pi_1(S)$ (or, equivalently, any curve $\gamma$ on $S$). 

Since $\nu$ is $\pi_1(S)$-invariant, $\pi_1(S)$ acts by isometries on $(X,d)$ and, since the action is cocompact on $\tilde{S}$ it is also cocompact on $(X,d)$. Moreover, it is not hard to see that $\pi_1(S)$ acts discretely on $(X,d)$ since $\nu$ is filling: if there exists a sequence $(\gamma_n)$ in $\pi_1(S)$ and $x\in\tilde{S}$ such that $d(x,\gamma_n(x))\to0$ as $n\to\infty$, then, by \ref{inequality}, $i(\nu, \gamma_n)\to0$ as $n\to\infty$, contradicting the fact that $\nu$ is filling.

We have the following result: 

\begin{theorem}
Let $\nu$ be any filling current. Then there exists a metric space $X$ on which $\pi_1(S)$ acts discretely and cocompactly by isometries such that
$$sl_X(\gamma)=i(\nu,\gamma)$$
for all curves $\gamma$ on $S$. 
\qed
\end{theorem}


\section{Applications to counting curves}\label{counting}

In \cite{Mir1, Mir2} Mirzakhani gives the asymptotic growth rate of the number of curves of bounded length, in each $\Mod(S)$-orbit, as the length grows. 

\begin{theorem}\cite[Theorem 1.1]{Mir1, Mir2}\label{Maryam}
Let $\gamma_0$ be a curve on $S$, and ${\rho}$ be a hyperbolic metric on $S$. Then 
$$\lim_{L\to\infty}\frac{\#\{\gamma\in\Mod(S)\cdot\gamma_0\,\vert\,\ell_{\rho}(\gamma)\leq L\}}{L^{6g-6}} = C_{\gamma_0}\cdot m_{\rho}$$
for some $C_{\gamma_0}>0$, and $m_{\rho}=m_{Th}(\{\lambda\in\ML(S)\,\vert\,\ell_{\rho}(\lambda)\leq1\})$ where $m_{Th}$ is the Thurston measure on $\ML(S)$. 
\end{theorem}

The constant $C_{\gamma_0}$ in Theorem \ref{Maryam} is independent of the hyperbolic metric ${\rho}$. In fact, Mirzakhani \cite{Mir1} showed that 
$$C_{\gamma_0}=\frac{n_{\gamma_0}}{\bold{m}_g}$$
where $n_{\gamma_0}>0$ depends only on $\gamma_0$ and 
\begin{equation}
\bold{m}_g = \int_{\mathcal{M}} m_{\rho} \, d\hspace{-0.1cm}\vol_{WP}\label{WP}
\end{equation}
where the integral is taken over the moduli space 
$$\mathcal{M}=\Teich(S)/\Mod(S)$$
with respect to the Weil--Petersson volume form. 

The \emph{Thurston measure} is the natural $\Mod(S)$-invariant locally finite measure on $\ML(S)$ given by the piecewise linear structure coming from train track coordinates. See section \ref{th} for details.  

The purpose of this section is to discuss a generalization of the theorem of Mirzakhani above, based on the previous section (see Theorem \ref{juanviveka}). We will explain why the same asymptotic behavior as in Theorem \ref{Maryam} holds for other metrics on $S$, in particular for any Riemannian metric. The results presented are contained in \cite{ES}, \cite{viv} and \cite{EPS}. The idea behind the proof of the generalization to other metrics crystallized over the above series of papers, so we provide a unified but brief explanation for the statements and proofs of these results. 

\begin{remark}
Theorem \ref{Maryam}, as well as its generalization Theorem \ref{juanviveka} below, holds for any finite type, orientable surface of negative Euler characteristic (other than the thrice punctured sphere). That is, we can allow $S$ to have $n$ punctures or boundary components, and the same asymptotic behavior holds (where we replace $6g-6$ in the exponent with $6g-6+2n$). However, somewhat surprisingly, orientability is a necessary condition.  For non-orientable surfaces the theorems fail, see \cite{Gendulphe, Magee}.
\end{remark}


\subsection{Thurston measure}\label{th}

Recall, from section \ref{laminations}, that the space $\ML(S)$ of measured laminations has a $\Mod(S)$-invariant piecewise linear manifold structure. Moreover, the $\PL$-manifold is equipped with a $\Mod(S)$-invariant symplectic structure, which gives rise to a $\Mod(S)$-invariant measure in the Lebesgue class. This is the Thurston measure $m_{Th}$. 
It is infinite, but locally finite, and satisfies
$$m_{Th}(L\cdot U)=L^{6g-6}\cdot m_{Th}(U)$$
for every Borel set $U\subset\ML(S)$ and $L>0$ (see \cite{ThurstonNotes}). Furthermore, as shown by Masur \cite{MasurPAMS}, the Thurston measure $m_{Th}$ is ergodic with respect to the $\Mod(S)$-action on $\ML(S)$, and is the only (up to scaling) invariant measure in the Lebesgue class. Recall that a measure $m$ is said to be ergodic with respect to $\Mod(S)$ if for every $\Mod(S)$-invariant Borel set $U$ we have that either $m(U)=0$ or $m(U^{c})=0$. 

In this section we explain how one can see the Thurston measure (up to scaling) as a limit of a sequence of measures, which gives perhaps a more intuitive feeling of what this measure is. 

For each $L$, define a measure on $\ML(S)$ by
$$m^{L} =\frac{1}{L^{6g-6}} \sum_{\gamma\in\ML_{\mathbb{Z}}(S)} \delta_{\frac{1}{L}\gamma}$$
where $\delta_x$ denotes the Dirac measure centered at $x$ and $\ML_{\mathbb{Z}}(S)$ is the subset of $\ML(S)$ corresponding to integral multicurves. We will show that, as $L\to\infty$, these measures converge to a multiple of the Thurston measure, i.e.

\begin{equation}\label{thurston measure}
\lim_{L\to\infty}\frac{1}{L^{6g-6}} \sum_{\gamma\in\ML_{\mathbb{Z}}(S)} \delta_{\frac{1}{L}\gamma} = c\cdot m_{Th}
\end{equation}
for some $c>0$. 
Note that each $m^{L}$ is $\Mod(S)$-invariant, and hence so is any limit. We will show that the limit is moreover in the Lebesgue class and it follows that it must be a multiple of the Thurston measure. 

It is enough to show the convergence of the measures in each chart given by the linear piecewise structure on $\ML(S)$. Hence we fix a maximal train track $\tau$ and let $C(\tau)$ be the solution set to the switch equations of $\tau$. The set $C(\tau)$ is a rational cone of dimension $6g-6$ in $\mathbb{R}^{E}$, where $E$ is the number of edges of $\tau$, and defines an open set in $\ML(S)$ given by all measured laminations carried by $\tau$. The integral simple multicurves carried by $\tau$ correspond exactly to the integer points in $C(\tau)$ which in turn, by the rationality of $C(\tau)$, we identify with a subset of $\mathbb{Z}^{6g-6}$. Accordingly, we identify $C(\tau)$ with a cone $C'(\tau)$ in $\mathbb{R}^{6g-6}$ such that $C'(\tau)\cap \mathbb{Z}^{6g-6}$ correspond to the integral multicurves carried by $\tau$. Finally, we push forward $m^L$ through these identifications to a measure on $\mathbb{R}^{6g-6}\cap C'(\tau)$ which is the restriction of the measure 
$$m_{\tau}^{L} = \frac{1}{L^{6g-6}} \sum_{p\in\mathbb{Z}^{6g-6}} \delta_{\frac{1}{L}p}$$
(viewed as a measure on $\mathbb{R}^{6g-6}$) to the cone $C'(\tau)$. 


It is not hard to see that $m_{\tau}^{L}$ converges to the Lebegue measure as $L\to\infty$. However, we include an outline for a proof of this statement here, since we will use a similar argument in section \ref{sec:countinglength} concerning convergence of a family of measures on the space of currents. 

Note that the family $(m_{\tau}^{L})_{L}$ is precompact in the space of Radon measures on $\mathbb{R}^{6g-6}$, meaning that any sequence of measures has a subsequence that weakly converges to a measure. Indeed, since the space of probability measures on a compact metric space is compact, it is enough to show that  
\begin{equation}\label{limsup}
\limsup_{L\to\infty} m_{\tau}^L(R_s)<\infty
\end{equation}
where $R_s$ is a (closed) cube of side length $s$ in $\mathbb{R}^{6g-6}$. Clearly we have 
\begin{equation}\label{eq:trivial}
(s-1)^{6g-6}\leq\#R_s\cap\mathbb{Z}^{6g-6}\leq(s+1)^{6g-6}
\end{equation}
and so 
\begin{equation*}\label{lebesgue}
m_{\tau}^L(R_s)=\frac{\#\{p\in\mathbb{Z}^{6g-6}\,\vert\,p\in R_{s\cdot L}\}}{L^{6g-6}}\leq\frac{(sL+1)^{6g-6}}{L^{6g-6}}
\end{equation*}
and the limit (superior) of the right hand side is finite. Hence \eqref{limsup} holds.
Now let $m$ be any limit point of $(m_{\tau}^{L})_{L}$.  Note that for each $L$, the measure $m_{\tau}^{L}$ is invariant under translation in the lattice $(\frac{1}{L}\mathbb{Z})^{6g-6}$. It follows that $m$ is translation invariant in $\mathbb{R}^{6g-6}$ and hence must be a multiple of the Lebesgue measure (since this is the unique measure, up to scaling, with this property). We have: For for any $(L_n)_{n}$ with $L_n\to \infty$ there exists a subsequence $(L_{n_k})_{k}$ such that
$$m_{\tau}^{L_{n_k}} \to c\cdot \mathfrak{L}$$ 
for some $c>0$ as $k\to\infty$, where $\mathfrak{L}$ denotes the Lebesgue measure. Hence, to prove \eqref{thurston measure} we need to show that $c$ is independent of the subsequence. Note that, as above, 
$$m_{\tau}^L(R_1)=\frac{\#\left\{p\in\mathbb{Z}^{6g-6}\,\vert\,p\in R_L\right\}}{L^{6g-6}}$$
and the right hand side converges as $L\to\infty$ by \eqref{eq:trivial} to 1, that is, to the Lebesgue measure of the unit cube $R_1$. Hence the limit of the right hand side does not depend on the subsequence and \eqref{thurston measure} follows.


\subsection{Counting with respect to length functions}\label{sec:countinglength}

Given a hyperbolic metric $\rho$ on $S$ and its corresponding Liouville current $L_{\rho}$, one can replace the length function $\ell_{\rho}(\cdot)$ in Theorem \ref{Maryam} with the intersection function $i(L_{\rho}, \cdot)$. In view of this, one can consider the following generalization of the limit appearing in the mentioned theorem:
\begin{equation}\label{limit}
\lim_{L\to\infty}\frac{\#\{\gamma\in\Mod(S)\cdot\gamma_0\,\vert\,i(\nu, \gamma)\leq L\}}{L^{6g-6}}
\end{equation}
where $\gamma_0$ is a curve on $S$ and $\nu$ is any filling current. (Note that we require $\nu$ to be filling to guarantee that there are only finitely many curves with bounded intersection number with $\nu$). In particular, by letting $\nu$ be a Liouville current for another metric, such as a variable negatively curved or Euclidean cone metric, this is equivalent to asking if the limit \eqref{limit} exists with respect to this metric. 

In \cite{ES} it was shown that limit \eqref{limit} exists for any filling current $\nu$, and in fact, more generally when the intersection function $i(\nu,\cdot)$ is replaced by any continuous length function $\ell(\cdot)$ defined on the space of currents. Recall that we say $\ell$ is a length function on $\Curr(S)$ if it is homogeneous and $\ell(\mu)\geq 0$ for all currents $\mu$ and $\ell(\mu)=0$ if and only if $\mu=0$. 

\begin{theorem}\cite{ES}\label{juanviveka}
Let $\ell:\Curr(S)\to\mathbb{R}$ be any continuous length function and $\gamma_0$ a curve on $S$. Then
\[
\lim_{L\to\infty}\frac{\#\{\gamma\in\Mod(S)\cdot\gamma_0\,\vert\,\ell(\gamma)\leq L\}}{L^{6g-6}}=C_{\gamma_0}\cdot m_{\ell}
\]
where $C_{\gamma_0}>0$ is the same constant as in Theorem \ref{Maryam}, and 
$$m_{\ell}=m_{Th}\left(\{\lambda\in\ML(S)\,\vert\,\ell(\lambda)\leq1\}\right).$$ 
\end{theorem}

Here we give an outline of the arguments involved in proving Theorem \ref{juanviveka}, and refer to \cite{ES} for the details. 

The main idea to prove the convergence of the limit 
\begin{equation}\label{limit f}
\lim_{L\to\infty}\frac{\#\{\gamma\in\Mod(S)\cdot\gamma_0\,\vert\,\ell(\gamma)\leq L\}}{L^{6g-6}}
\end{equation}
is to consider a sequence of measures on $\Curr(S)$ analogous to the measures on $\ML(S)$ in section \ref{th}. Let $\gamma_0\in S$ be a curve and define, for each $L>0$, a measure on $\Curr(S)$ by
$$m_{\gamma_0}^{L}=\frac{1}{L^{6g-6}}\sum_{\gamma\in\Mod(S)\cdot\gamma_0}\delta_{\frac{1}{L}\gamma}.$$
Note that each $m_{\gamma_0}^L$ is locally finite and invariant under the action of $\Mod(S)$. In fact, we will see that, as $L\to\infty$ they converge to a $\Mod(S)$-invariant measure on $\ML(S)$ that is absolutely continuous with respect to the Thurston measure, and hence, using the ergodicity of $m_{Th}$, they must converge to a multiple of this measure:

\begin{theorem}\cite[Theorem 5.1]{ES, EPS}\label{thm convergence}
Let $\gamma_0$ be any curve on $S$. Then 
$$\lim_{L\to\infty} m^L_{\gamma_0} = C_{\gamma_0}\cdot m_{Th}$$
where $C_{\gamma_0}>0$ is the constant in Theorem \ref{Maryam}.  
\end{theorem} 

First we explain why Theorem \ref{thm convergence} implies Theorem \ref{juanviveka}. Fix a continuous length function $\ell: \Curr(S)\to\mathbb{R}$  and let 
$$B_{\ell}=\{\mu\in\Curr(S)\,\vert\,\ell(\mu)\leq1\}.$$
Note that limit \eqref{limit f} is equivalent to 
$$\lim_{L\to\infty}m_{\gamma_0}^L(B_{\ell}).$$ 
The continuity of $\ell$ implies that $B_{\ell}$ is a closed set. Also, for any measurable set $U$ satisfying $U\cap L\cdot U=\emptyset$ for any positive $L\neq1$, the scaling properties of the Thurston measure imply that $m_{Th}(U)=0$. To see this, note that for all $L\neq1$
$$m_{Th}(U\cup L\cdot U)=m_{Th}(U)+m_{Th}(L\cdot U)=m_{Th}(U)(1+L^{6g-6})$$
and letting $L\to1$ we get $m_{Th}(U)=2m_{Th}(U)$, i.e. $m_{Th}(U)=0$. In particular, $m_{Th}(\partial B_{\ell})=0$. Hence, by the Portmanteau Theorem, see \cite{Bil},
$$\lim_{L\to\infty} m^L_{\gamma_0} = C_{\gamma_0}\cdot m_{Th}$$
implies that
$$\lim_{L\to\infty}m_{\gamma_0}^L(B_{\ell})=C_{\gamma_0}\cdot m_{Th}(B_{\ell})$$
where we view $m_{Th}$ as a measure on $\Curr(S)$ with full support on the subspace $\ML(S)$. Theorem \ref{juanviveka} follows.  

Next we outline the arguments proving Theorem \ref{thm convergence}. In an attempt to aid the reader we first outline the main steps involved in the proof: 
\begin{enumerate}
\item Let $m_{\gamma_0}$ be any limit point of the family $(m_{\gamma_0}^L)_{L}$, and note that it is $\Mod(S)$-invariant. 
\item We show that $m_{\gamma_0}$ is supported on $\ML(S)$, and \label{ml}
\item that $m_{\gamma_0}$ is absolutely continuous with respect to the Thurston measure $m_{Th}$ on $\ML(S)$.\label{abs}
\item Ergodicity of $m_{Th}$ with respect to $\Mod(S)$ together with the steps above, imply that $m_{\gamma_0}=C\cdot m_{Th}$ for some $C>0$.\label{step multiple}
\item Finally, using Mirzakhani's theorem (Theorem \ref{Maryam}) we show that the constant $C$ above does not depend on the subsequence and is in fact equal to $C_{\gamma_0}$. Hence $m^{L}_{\gamma_0}\to C_{\gamma_0}\cdot m_{Th}$.\label{stepmaryam}
\end{enumerate}

We formalize the conclusion of step \ref{step multiple} below:

\begin{proposition}\cite[Proposition 4.1]{ES}\label{multiple}
Let $(L_n)_n$ be any sequence of positive numbers such that $L_n\to\infty$. Then there is a subsequence $(L_{n_k})_k$ such that 
$$m_{\gamma_0}^{L_{n_k}}\to C\cdot m_{Th}$$
for some $C>0$, as $k\to\infty$. 
\end{proposition}
As above, due to the Portmanteau theorem, we get the following consequence: 

\begin{corollary}\label{strong}
Let $\ell:\Curr(S)\to\mathbb{R}$ be a continuous length function and let $(L_n)_n$ be any sequence of positive numbers such that $L_n\to\infty$. Then there is a subsequence $(L_{n_k})_k$ such that 
$$m_{\gamma_0}^{L_{n_k}}(B_{\ell})\to C\cdot m_{Th}(B_{\ell})$$
for some $C>0$, as $k\to\infty$. 
\end{corollary}

The key idea behind proving Proposition \ref{multiple} is to associate to each (generic) curve in $\Mod(S)\cdot\gamma_0$ a simple multi-curve. Specifically, we define a map
$$\pi_{\gamma_0}^{\epsilon}: \Sigma_{\gamma_0}^{\epsilon}\to\ML_{\mathbb{Z}}(S)$$
where $\Sigma_{\gamma_0}^{\epsilon}\subset\Mod(S)\cdot\gamma_0$ is a generic subset 
such that 
\begin{equation}\label{epsilon}
(1-\epsilon)\ell(\gamma)<\ell(\pi^{\epsilon}_{\gamma_0}(\gamma))<(1+\epsilon)\ell(\gamma).
\end{equation}
We say a set $\Sigma$ is generic if 
$$\frac{\#\Sigma}{L^{6g-6}}\to 0$$
as $L\to\infty$. 
The existence of such a map results from the following observation, which says that the expected angle of self-intersection of a long curve is arbitrarily small. 

\begin{theorem}\cite[Theorem 1.2]{ES}\label{angle}
Let $\angle(\gamma)$ denote the largest angle among the self-intersection angles of a curve $\gamma$. Let $\gamma_0\subset S$ be a curve and $\rho$ a hyperbolic metric. Then 
$$\lim_{L\to\infty}\frac{\#\{\gamma\in\Mod(S)\cdot\gamma_0\,\vert\,\ell_{\rho}(\gamma)\leq L, \angle(\gamma)\geq\delta\}}{L^{6g-6}}=0$$
for all $\delta>0$.
\end{theorem}

The proof of Theorem \ref{angle} is quite involved (see \cite[Section 3.3]{ES}), but the general idea
is that large self-intersection angles result in ideal $4$-gons on the surface which most of the curves have to avoid. The set of curves on $S$ which do not intersect a $4$-gon must live on a proper subsurface and hence the number of these curves of length bounded by $L$ must grow at a slower rate than $L^{6g-6}$. This idea is inspired by the fact that the subspace of $\ML(S)$ of measured laminations carried by non-maximal train tracks (i.e. train tracks that have complementary regions larger than triangles) has dimension strictly less than $6g-6$.

Armed with Theorem \ref{angle}, we can resolve the self-intersections and end up with a simple multi-curve whose length is close to the length of the original curve (see \cite[Section 3.4]{ES} for details), and this is the idea for the map $\pi_{\gamma_0}^{\epsilon}$. In particular, for any $\epsilon>0$ there is an angle bound $\delta>0$ such that any curve $\gamma$ with self-intersection angles less than $\delta$ is mapped to a simple multi-curve $\pi^{\epsilon}_{\gamma_0}(\gamma)$ satisfying \eqref{epsilon}. These curves are what make up the generic set $S_{\gamma_0}^{\epsilon}$. 

We fix $\epsilon>0$ and suppress the superscript in $\pi^{\epsilon}_{\gamma_0}$ for ease of notation. It is clear that $\pi_{\gamma_0}$ is finite-to-one, but the main useful property of the map, and the key technical difficulty of the proof (details of which will be omitted here, see \cite[Section 2.4]{ES}) is that it is uniformly bounded-to-$1$. That is: 

\begin{lemma}\cite[Proposition 3.9]{ES}\label{bound}
There exists a constant $K=K(\gamma_0)>0$ such that
\begin{equation}\label{boundeq}
\vert\pi_{\gamma_0}^{-1}(\lambda)\vert<K
\end{equation}
for all $\lambda\in\ML_{\mathbb{Z}}(S)$.
\end{lemma}

We note that any limit point $m_{\gamma_0}$ is locally finite and $\Mod(S)$-invariant since this is true for each $m_{\gamma_0}^L$. We then use Lemma \ref{bound} to show that any limit point is also uniformly continuous with respect to the Thurston measure. To do so, we first push forward the measure $m_{\gamma_0}^L$ via $\pi_{\gamma_0}$ resulting in the following measure supported on $\ML(S)$:
 $$n_{\gamma_0}^L = \frac{1}{L^{6g-6}}\sum_{\lambda\in\ML_{\mathbb{Z}}(S)}\vert\pi_{\gamma_0}^{-1}(\lambda)\vert\delta_{\frac{1}{L}\lambda}.$$
It is not difficult to see $m_{\gamma_0}$ is a limit point of the family $(m_{\gamma_0}^L)_{L}$ if an only if it is a limit point of the family $(n_{\gamma_0}^L)_{L}$. In particular, any limit point is supported on $\ML(S)$, completing step \ref{ml}. 
 
Now, \eqref{boundeq} implies that
$$n_{\gamma_0}^L =  \frac{1}{L^{6g-6}}\sum_{\lambda\in\ML_{\mathbb{Z}}(S)}\vert\pi_{\gamma_0}^{-1}(\lambda)\vert\delta_{\frac{1}{L}\lambda} < K\cdot \frac{1}{L^{6g-6}}\sum_{\lambda\in\ML_{\mathbb{Z}}(S)}\delta_{\frac{1}{L}\lambda}$$
and the right hand side converges to a multiple of $m_{Th}$ as $L\to\infty$ (see \eqref{thurston measure}). In particular, any limit point of $(n_{\gamma_0}^L)_{L}$, and hence of $(m_{\gamma_0}^L)_{L}$, is absolutely continuous with respect to the Thurston measure, completing step \ref{abs}. 
 
Next, recall that, by a result of Masur \cite{MasurPAMS}, the Thurston measure is ergodic with respect to the action of $\Mod(S)$ on $\ML(S)$. Hence, since any limit $m_{\gamma_0}$ of $(m_{\gamma_0}^L)_{L}$ is invariant under this action and absolutely continuous with respect to the Thurston measure, the only choice for $m_{\gamma_0}$ is a positive multiple of the Thurston measure. This completes the argument for proving Proposition \ref{multiple} (and hence step \ref{step multiple}). 

Finally, we use Mirzakhani's result (Theorem \ref{Maryam}) to complete the outline of the proof of Theorem \ref{thm convergence}. We need to show that the constant $C$ in Proposition \ref{multiple} is independent of the subsequence and that $C$ is in fact equal to the constant $C_{\gamma_0}$.

Let $\rho$ be a hyperbolic metric on $S$ and $L_{\rho}$ the corresponding Liouville current. Let $\ell_{\rho}:\Curr(S)\to\mathbb{R}$ be the length function defined by 
$$\ell_{\rho}(\mu)=i(\mu, L_{\rho})$$
which agrees with the hyperbolic length on curves. Following the notation above, let
$$B_{\ell_{\rho}}=\{\mu\in\Curr(S)\,\vert\,i(\mu,L_{\rho})\leq1\}.$$
By definition, 
$$m_{\gamma_0}^{L}(B_{\ell_{\rho}}) = \frac{\#\{\gamma\in\Mod(S)\cdot\gamma_0\,\vert\,\ell_{\rho}(\gamma)\leq L\}}{L^{6g-6}}.$$
By Theorem \ref{Maryam} we know that the right hand side converges to 
$$C_{\gamma_0}\cdot m_{Th}(B_{\ell_{\rho}}).$$
In particular, $m_{\gamma_0}^L(B_{\ell_{\rho}})$ converges and by Corollary \ref{strong} it must converge to $C\cdot m_{Th}(B_{\ell_{\rho}})$ for some $C>0$. Hence we have $C=C_{\gamma_0}$, completing step \ref{stepmaryam}, and Theorem \ref{thm convergence} follows.


Lastly, we note tells in particular that we have the asymptotic growth
$$\#\{\gamma\in\Mod(S)\cdot\gamma_0\,\vert\,\ell(\gamma)\leq L\} \sim const\cdot L^{6g-6}$$
for any of the length functions $\ell$ discussed in section \ref{length}. In particular, it is true for the length induced by any non-positive (singular) Riemannian metric on $S$ as well as for the stable length with respect to a geodesic metric space $X$ on which $\pi_1(S)$ acts discretely and cocompactly by isometries (see Theorem \ref{stable}). However, in \cite{EPS} it is shown that it is enough for the stable length to extend to $\Curr(S)$ to conclude that the asymptotics above hold for the actual (translation) length. In particular, it holds for \emph{any} Riemannian metric on $S$.

\begin{corollary}\cite[Corollaries 1.3 and 1.4]{EPS}\label{EPStheorem}
Let $\gamma_0$ be a curve on $S$. If $\rho$ is any (possibly singular) Riemannian metric on $S$ and $\ell_{\rho}$ is the corresponding length function on curves, then 
$$ \lim_{L\to\infty}\frac{\#\{\gamma\in\Mod(S)\cdot\gamma_0\,\vert\,\ell_{\rho}(\gamma)\leq L\}}{L^{6g-6}}$$
exists and is positive. Similarly, if we replace the length function with the word length with respect to any finite generating set of $\pi_1(S)$ then the corresponding limit also exists and is positive. 
\end{corollary}   


\subsection{Orbits of currents}

We end by remarking that one could also ask whether the limit in Theorem \ref{juanviveka} exists if we look at the $\Mod(S)$-orbit of any \emph{current} instead of a curve. Rafi-Souto proved that this is indeed the case:

\begin{theorem}\cite[Main Theorem]{RS}\label{Rafi}
Let $\ell:\Curr(S)\to\mathbb{R}$ be a continuous length function. For any filling current $\nu\in\Curr(S)$ we have
$$\lim_{L\to\infty}\frac{\#\{\mu\in\Mod(S)\cdot\nu\,\vert\,\ell(\mu)\leq L\}}{L^{6g-6}}=C_{\nu}\cdot m_{\ell}$$
where $C_{\nu}>0$ and $m_{\ell}=m_{Th}(\{\lambda\in\ML\,\vert\,\ell(\lambda)\leq1\})$.
\end{theorem}

The constant $C_{\nu}$, as in Theorem \ref{Maryam} is independent of $\ell$ and can be written as
$$C_{\nu}=\frac{n_{\nu}}{\bold{m}_g}$$
where $\bold{m}_g$ is the same constant as in \eqref{WP}. However, in \cite{RS} the constant $n_{\nu}$, in the case when $\nu$ is filling, is also described:
$$n_{\nu}=m_{Th}(\{\lambda\in\ML\,\vert\,i(\nu,\lambda)\leq1\}).$$

The proof of Theorem \ref{Rafi} follows a similar logic to the proof of Theorem \ref{juanviveka} above. However, in order to generalize Proposition \ref{multiple} to hold also when $\gamma_0$ is a filling current, they combine this proposition together with a deep result of Lindenstrauss-Mirzakhani \cite{LM} about the classifications of invariant measures on $\ML(S)$.

We note that Theorem \ref{Rafi} holds also for surfaces with boundary, as do Theorems \ref{Maryam} and \ref{juanviveka}, but unlike the latter two which also work for surfaces with cusps, Theorem \ref{Rafi} requires $S$ to be compact (or alternatively, that $\nu$ has compact support). 

As an application to Theorem \ref{Rafi}, Rafi-Souto prove the asymptotic growth of lattice points in Teichm\"{u}ller space with respect to the Thurston metric. As before, for a length function $f:\Curr(S)\to\mathbb{R}$ we let $m_f$ denote the constant
$$m_f=m_{Th}(\{\lambda\in\ML(S)\,\vert\,f(\lambda)\leq1\})$$
and we let $m_X$ denote the corresponding constant when $f=\ell_X$, the hyperbolic length on $X\in\Teich(S)$. 

\begin{theorem}\cite[Theorem 1.1]{RS}\label{Rafi2}
Let $X, Y\in\Teich(S)$. Then 
$$\lim_{R\to\infty}\frac{\#\{\varphi\in\Mod(S)\,\vert\,d_{Th}(X,\varphi(Y))\leq R\}}{e^{(6g-6)R}}=\frac{m_{D_X}m_Y}{\bold{m}_g}$$
where $d_{Th}$ denotes the Thurston metric on $\Teich(S)$, $\bold{m}_g$ is as above, and $$D_X(\mu)=\max_{\lambda\in\ML(S)}\frac{i(\lambda,\mu)}{\ell_X(\lambda)}.$$
\end{theorem}

The analogous result of Theorem \ref{Rafi2} when the Thurston metric is replaced by the Teichm\"{u}ller metric was proved, using different methods, by Athreya-Bufetov-Eskin-Mirzakhani in \cite{ABEM}.






\section{Dynamics of pseudo-Anosov homeomorphisms}\label{NSDsection}

The purpose of this section is to give a concise proof of a folklore result using Bonahon's intersection function on the space of currents: pseudo-Anosov homeomorphisms of closed hyperbolic surfaces act on the space of projective geodesic currents with uniform north-south dynamics. 

\begin{theorem}\label{NSD} Let $S$ be closed hyperbolic surface and $\varphi:S\to S$ be a pseudo-Anosov homeomorphism. Then $\varphi$ acts on the space of projective geodesic currents $\PCurr(S)$ with uniform north-south dynamics: The action of 
$\varphi$ on $\PCurr(S)$ has exactly two fixed points $[\lambda_{+}]$ and $[\lambda_{-}]$ and for any open neighborhood $U_{\pm}$ of $[\lambda_{\pm}]$ and a compact set $K_{\pm}\subset \PCurr(S)\setminus[\lambda_{\mp}]$, there exist an exponent $M\ge1$ such that $\varphi^{\pm n}(K_{\pm})\subset U_{\pm}$ for all $n\ge M$. 

\end{theorem}

The idea of the proof is as follows: The set of non-zero currents that have zero intersection with the \emph{stable} current/lamination is precisely the positive scalar multiples of the \emph{stable} current/lamination. Similarly, the set of non-zero currents that has zero intersection with the \emph{unstable} current/lamination is precisely the positive scalar multiples of the \emph{unstable} current/lamination, see Lemma \ref{ergodic}.  

Using Lemma \ref{ergodic} we define continuous functions $J_{+}$ and $J_{-}$ on the space of projective currents which take the value $0$ only on $[\lambda_{+}]$ and [$\lambda_{-}]$ respectively. We then use these functions to construct neighborhoods of $[\lambda_{+}]$ and [$\lambda_{-}]$ and use the properties of intersection function to get convergence estimates. 
 
The proof we present here is motivated by Ivanov's proof of north-south dynamics in the setting of projective measured laminations \cite{Iva}, and consists of putting together a series of lemmas, which we first state and prove.


\begin{lemma}\label{ergodic} Let $\varphi:S\to S$ be a pseudo-Anosov homeomorphism on a closed hyperbolic surface and $\lambda_{+}$ and $\lambda_{-}$ be the corresponding stable and unstable laminations for $\varphi$. Then,
\[
i(\lambda_{\pm},\mu)=0\ \text{if and only if}\ \mu=k\lambda_{\pm}
\] for some $k\ge0$.  
\end{lemma}

\begin{proof} Here we give a brief idea of the proof and refer the reader to proof of \cite[Proposition 3.1]{UyaNSD} for details in the case of non-closed surfaces, where the proof is more involved.  Let $\lambda_{+}$ be the stable lamination on $S$ corresponding to the pseudo-Anosov homeomorphism $f$. The proof for $\lambda_{-}$ is almost  identical. 

We first prove the easy direction of the statement. Namely, let $\mu=k\lambda_{+}$, and $\alpha>1$ be such that $\varphi(\lambda_{+})=\alpha \lambda_{+}$. Then, by properties of the intersection number we have
\begin{align*} 
i(k\lambda_{+},\lambda_{+})=i(\varphi^{n}(k\lambda_{+}),\varphi^{n}(\lambda_{+}))&\\
&= i(\alpha^nk\lambda_+,\alpha^n\lambda_+)\\
&=\alpha^{2n}i(k\lambda_{+},\lambda_{+}) \\
\end{align*} 
which implies $i(k\lambda_{+},\lambda_{+})=0$. 

For the forward implication, we first cut the surface along the leaves of the stable lamination. The complementary regions are finite sided ideal polygons, \cite[Proposition 5.3]{CB}. Let $\mu$ be any current such that $i(\mu,\lambda_{+})=0$. Let $\ell$ be any leaf in the support of $\mu$, since the projection of this leaf onto the surface cannot intersect the leaves of the stable lamination transversely, there are two possibilities for this projection. Either $\ell$ projects onto a leaf of the lamination $\lambda$ or it is a complete geodesic that is asymptotic to two different sides of a complementary polygon. In the second case, this leaf cannot support any measure, otherwise the corresponding current would not be locally finite. Hence $\mu$ and $\lambda_{+}$ have the same support, and unique ergodicity of $\lambda_{+}$ implies that $\mu=k\lambda_{+}$. 

\end{proof}

Fix a filling current $\nu$ on $S$, and consider the following two functions $J_{+},J_{-}:\PCurr(S)\to\RR_{\ge0}$
defined by 

\[
J_{+}([\mu])=\frac{i(\mu,\lambda_{+})}{i(\mu,\nu)}\ ,\ \  J_{-}[\mu]=\frac{i(\mu,\lambda_{-})}{i(\mu,\nu)}
\]
where $\mu$ is any representative of $[\mu]$. Note that $J_{+},J_{-}$ are well defined and continuous since the intersection function is continuous and homogeneous, and the  denominator is non-zero by the choice of $\nu$. 
\begin{lemma}\label{south}
Let $\alpha$ be the stretch factor for the pseudo-Anosov element $\varphi$ and let $\nu$ be a filling current. If $K$ is a compact set in $\PCurr(S)\setminus [\lambda_{-}]$, then there exist $C>0$ such that 
\[
\frac{1}{i(\varphi^{n}(\mu),\nu)}\leq\frac{C}{\alpha^{n}i(\mu,\nu)}
\]
for all $\mu$ such that $[\mu]\in K$. \

\end{lemma}

\begin{proof}
Since $\PCurr(S)$ is compact, there exist $0<C_1<\infty$ such that 
\[
J_{-}([\mu])=\frac{i(\mu,\lambda_{-})}{i(\mu,\nu)}\leq C_{1}
\]
i.e. 
\[
i(\mu,\lambda_{-})\leq C_{1}i(\mu,\nu)
\]
for all nonzero $\mu\in \Curr(S)$.  

Furthermore, by Lemma \ref{ergodic} the quantity $i(\mu,\lambda_{-})$ is non-zero for any $\mu$ such that $[\mu]\in\PCurr(S)\setminus [\lambda_{-}]$. Therefore, by compactness of $K$, there exist $C_2>0$ such that  
\[
i(\mu,\lambda_{-})\geq C_{2}i(\mu,\nu)
\]
for all $\mu$ with $[\mu]\in K$.

From these two inequalities we obtain, for all $\mu$ such that $[\mu]\in K$,

\begin{align*} 
i(\varphi^{n}(\mu),\nu) \ge \dfrac{1}{C_1}i(\varphi^{n}(\mu),\lambda_{-})&=  \dfrac{1}{C_1}i(\mu, \varphi^{-n}(\lambda_{-}))\\
&= \dfrac{1}{C_1}i(\mu, \alpha^{n}\lambda_{-})\\
&=\dfrac{1}{C_1}\alpha^{n}i(\mu, \lambda_{-})\\
&\ge \dfrac{C_2}{C_1}\alpha^{n}i(\mu,\nu).
\end{align*}

Setting $C=\dfrac{C_{1}}{C_{2}}$, the conclusion of the lemma follows. 

\end{proof}

\begin{lemma}\label{north}
Let $U$ be an open neighborhood of $[\lambda_{+}]$ and $K$ be a compact
set in $\PCurr(S)\setminus[\lambda_{-}]$. There exist $M_1>0$ such that 
\[
\varphi^{n}(K)\subset U
\]
for all $n\ge M_1$.  
\end{lemma}

\begin{proof}

Since $i(\mu,\lambda_{+})=0$ iff $[\mu]=[\lambda_{+}]$, and $\PCurr(S)\setminus U$ is compact, the function $J_{+}([\mu])$ has a positive absolute minimum on the set $\PCurr(S)\setminus U$, say $\epsilon>0$. Therefore, it suffices to prove that $J_{+}(\varphi^{n}[\mu])<\epsilon$ for all $[\mu]\in K$, and for all large $n$ in order to obtain the conclusion of the lemma.

On the other hand, $\PCurr(S)$ is compact, so the function $J_{+}([\mu])$
has an upper bound, i.e. there exists $0<D<\infty$ such that
\[
\frac{i(\mu,\lambda_{+})}{i(\mu,\nu)}\leq D
\]
for all $\mu$. 

Let $\epsilon>0$ be as above, and choose $M_1>0$ such that $\dfrac{DC}{\alpha^{2M_1}}<\epsilon$ where $\alpha$ is the stretch factor of $\varphi$ and $C$ is the constant given by Lemma \ref{south}. Then, for all $[\mu]\in K$ we have 

\begin{align*} 
J_{+}(\varphi^{n}[\mu])=\dfrac{i(\varphi^{n}(\mu), \lambda_{+})}{i(\varphi^{n}(\mu),\nu)}&=\dfrac{i(\mu, \varphi^{-n}(\lambda_{+}))}{i(\varphi^{n}(\mu),\nu)}\\
&= \dfrac{\alpha^{-n}i(\mu, \lambda_{+})}{i(\varphi^{n}(\mu),\nu)}\\
&\le \dfrac{C \alpha^{-n}i(\mu,\lambda_{+})}{\alpha^{n}i(\mu,\nu)}\\
& \le \dfrac{DC}{\alpha^{2n}}<\epsilon
\end{align*} 
for all $n\ge M_1$. 

\end{proof}

We are now ready to prove the theorem: 

\begin{proof}[Proof of Theorem \ref{NSD}] Using an argument symmetric to the one in the proof of Lemma \ref{north}, we can show that 
given any compact set $ K\subset \PCurr(S)\setminus [\lambda_{-}]$ and an open neighborhood $U$ of $[\lambda_{-}]$ there exist $M_2>0$ such that $\varphi^{-n}(K)\subset U$ for all $n\ge M_2$. The theorem now follows by setting $M=\max\{M_1,M_2\}$. 

\end{proof}

In fact, we have much more precise information in terms of pointwise dynamics: 

\begin{theorem}\label{pointwise} Let $\alpha>1$ be the stretch factor for $\varphi$. Then, for any $[\mu]\neq[\lambda_{-}]$, 
\[
\lim_{n\to\infty}\alpha^{-n}\varphi^{n}(\mu)=c_{\mu}\lambda_{+}
\] for some $c_\mu>0$, 
and for any $[\mu']\neq[\lambda_{+}]$
\[
\lim_{n\to\infty}\alpha^{-n}\varphi^{-n}(\mu')=c_{\mu'}\lambda_{-}
\]
for some $c_{\mu'}>0$. 
\end{theorem}

\begin{proof} The proof builds on the analogous result in the case of laminations, and nearly identical to the case where $S$ has boundary components, see the proof of \cite[Theorem 3.4]{UyaNSD}. 
\end{proof}

Recall, from section \ref{section:Liouville}, that given a hyperbolic metric $\rho$ on $S$ the hyperbolic length extends to a continuous length function $\ell$ on $\Curr(S)$ given by
$$\ell_{\rho}(\mu)=i(L_{\rho},\mu)$$
for all $\mu\in\Curr(S)$, where $L_{\rho}$ is the Liouville currents associated to $\rho$. As an application to this we get as a corollary to the north-south dynamics the following generalization, to \emph{all} curves, of a well known result about simple closed curves: 

 

\begin{corollary}\label{futer} 
For any pseudo-Anosov homeomorphism $\varphi:S\to S$ of a closed, orientable hyperbolic surface $S$, there exists $M>0$ such that for any essential (not necessarily simple) closed curve $\gamma$ on $S$, either 
\[
\ell_{\rho} (\varphi^{k}\gamma)> \ell_{\rho}(\gamma)\ \text{  or  }\ \ell_{\rho} (\varphi^{-k}\gamma)> \ell_{\rho}(\gamma)
\]
for all $k\ge M$. 
\end{corollary}

\begin{proof}  We will show that there exists $M>0$ such that for all $\gamma$ on $S$, the following holds:
\[
\dfrac{\ell (\varphi^{k}\gamma)+\ell (\varphi^{-k}\gamma)}{\ell(\gamma)}>2
\]
for all $k\geq M$. 
Let $L_{\rho}$ be the Liouville current associated to the hyperbolic metric $\rho$. Note that, for all $k$, 
\[
\ell_{\rho}(\varphi^{k}\gamma)=i(\varphi^{k}\gamma,L_{\rho})=i(\gamma,\varphi^{-k}L_{\rho})
\]
and 
\[
\ell_{\rho}(\varphi^{-k}\gamma)=i(\varphi^{-k}\gamma,L_{\rho})=i(\gamma,\varphi^{k}L_{\rho}). 
\]
Hence it suffices to prove that there exist $M>0$ such that 
\begin{equation}\label{>2}
\dfrac{i(\gamma,\varphi^{k}L_{\rho})+i(\gamma,\varphi^{-k}L_{\rho})}{\ell_{\rho}(\gamma)}>2
\end{equation}
for all $\gamma$ and for all $k\geq M$.


Let $\alpha>1$ be the stretch factor for $\varphi$ and $\lambda_+$ and $\lambda_-$ its stable and unstable laminations, respectively. Using the properties of the intersection form we have
\begin{align*}
\dfrac{i(\gamma_k,\varphi^{k}L_{\rho})+i(\gamma_k,\varphi^{-k}L_{\rho})}{\ell_{\rho}(\gamma_k)}&=
i\left(\dfrac{\gamma_k}{\ell_{\rho}(\gamma_k)}, \varphi^{k}L_{\rho}\right)+i\left(\dfrac{\gamma_k}{\ell_{\rho}(\gamma_k)}, \varphi^{-k}L_{\rho}\right)\\&= \alpha^{k}i\left(\dfrac{\gamma_k}{\ell_{\rho}(\gamma_k)}, \alpha^{-k} \varphi^{k}L_{\rho}\right)+\alpha^{k}i\left(\dfrac{\gamma_k}{\ell_{\rho}(\gamma_k)}, \alpha^{-k}\varphi^{-k}L_{\rho}\right)
\end{align*}

Since the length of $\dfrac{\gamma_k}{\ell_{\rho}(\gamma_k)}$ is $1$ for all $k$, they lie in a compact set and hence there exist $\mu\in\Curr(S)$ such that (up to passing to a subsequence) 
\[
\lim_{k\to\infty}\frac{\gamma_k}{\ell(\gamma_k)}=\mu.
\]

On the other hand, Theorem \ref{pointwise} implies that, for some $c_0, c_1>0$ 
\[
\lim_{k\to\infty}\alpha^{-k}\varphi^{k}L_{\rho}=c_0 \lambda_{+}
\ \  \text{and}\ \ 
\lim_{k\to\infty}\alpha^{-k}\varphi^{-k}L_{\rho}=c_1 \lambda_{-}. 
\]

Therefore, as $k\to\infty$, 
\[
i\left(\dfrac{\gamma_k}{\ell_{\rho}(\gamma_k)}, \alpha^{-k} \varphi^{k}L_{\rho}\right)\to c_0i(\mu,\lambda_{+})\ \ \text{and}\ \ 
i\left(\dfrac{\gamma_k}{\ell_{\rho}(\gamma_k)}, \alpha^{-k} \varphi^{-k}L_{\rho}\right)\to c_1i(\mu,\lambda_{-}). 
\]

Since $\{\lambda_{+},\lambda_{-}\}$ is a filling pair of currents, at least one of the quantities $i(\mu,\lambda_{+})$ or $i(\mu,\lambda_{-})$ must be positive. Hence, since $\alpha^{k}\to\infty$ as $k\to\infty$, 
we have
\[
\dfrac{i(\gamma_k,\varphi^{k}L_{\rho})+i(\gamma_k,\varphi^{-k}L_{\rho})}{\ell_{\rho}(\gamma_k)}\to\infty.
\]

In particular, there exists $M>0$ such that \eqref{>2} holds for all $k\geq M$. 


\end{proof}



\bibliography{bib}

\begin{thebibliography}{ABEM12}

\bibitem[ABEM12]{ABEM}
Jayadev Athreya, Alexander Bufetov, Alex Eskin, and Maryam Mirzakhani.
\newblock Lattice point asymptotics and volume growth on {T}eichm\"uller space.
\newblock {\em Duke Math. J.}, 161(6):1055--1111, 2012.

\bibitem[AL]{AL}
J~Aramayona and C~Leininger.
\newblock Hyperbolic structures on surfaces and geodesic currents, to appear in
  algorithms and geometric topics around automorphisms of free groups, advanced
  courses crm.

\bibitem[Bil99]{Bil}
Patrick Billingsley.
\newblock {\em Convergence of probability measures}.
\newblock Wiley Series in Probability and Statistics: Probability and
  Statistics. John Wiley \& Sons, Inc., New York, second edition, 1999.
\newblock A Wiley-Interscience Publication.

\bibitem[BL18]{BL}
Anja Bankovic and Christopher~J. Leininger.
\newblock Marked-length-spectral rigidity for flat metrics.
\newblock {\em Trans. Amer. Math. Soc.}, 370(3):1867--1884, 2018.

\bibitem[Bon86]{Bo86}
Francis Bonahon.
\newblock Bouts des vari\'et\'es hyperboliques de dimension {$3$}.
\newblock {\em Ann. of Math. (2)}, 124(1):71--158, 1986.

\bibitem[Bon88]{Bo88}
Francis Bonahon.
\newblock The geometry of {T}eichm\"uller space via geodesic currents.
\newblock {\em Invent. Math.}, 92(1):139--162, 1988.

\bibitem[Bon91]{Bo91}
Francis Bonahon.
\newblock Geodesic currents on negatively curved groups.
\newblock In {\em Arboreal group theory ({B}erkeley, {CA}, 1988)}, volume~19 of
  {\em Math. Sci. Res. Inst. Publ.}, pages 143--168. Springer, New York, 1991.

\bibitem[CB88]{CB}
Andrew~J. Casson and Steven~A. Bleiler.
\newblock {\em Automorphisms of surfaces after {N}ielsen and {T}hurston},
  volume~9 of {\em London Mathematical Society Student Texts}.
\newblock Cambridge University Press, Cambridge, 1988.

\bibitem[Con18]{Constantine}
David Constantine.
\newblock {M}arked length spectrum rigidity in nonpositive curvature with
  singularities.
\newblock {\em to appear in Indiana University Mathematics Journal}, 2018.

\bibitem[DLR10]{DLR}
Moon Duchin, Christopher~J. Leininger, and Kasra Rafi.
\newblock Length spectra and degeneration of flat metrics.
\newblock {\em Invent. Math.}, 182(2):231--277, 2010.

\bibitem[EPS16]{EPS}
Viveka Erlandsson, Hugo Parlier, and Juan Souto.
\newblock Counting curves, and the stable length of currents.
\newblock {\em to appear in JEMS,
  \href{https://arxiv.org/abs/1612.05980}{arxiv:1612.05980}}, 2016.

\bibitem[Erl16]{viv}
Viveka Erlandsson.
\newblock A remark on the word length in surface groups.
\newblock {\em to appear TAMS,
  \href{https://arxiv.org/abs/1608.07436}{arXiv:1608.07436}}, 2016.

\bibitem[ES16]{ES}
Viveka Erlandsson and Juan Souto.
\newblock Counting curves in hyperbolic surfaces.
\newblock {\em Geom. Funct. Anal.}, 26(3):729--777, 2016.

\bibitem[FLP12]{FLP}
Albert Fathi, Fran\c{c}ois Laudenbach, and Valentin Po\'enaru.
\newblock {\em Thurston's work on surfaces}, volume~48 of {\em Mathematical
  Notes}.
\newblock Princeton University Press, Princeton, NJ, 2012.
\newblock Translated from the 1979 French original by Djun M. Kim and Dan
  Margalit.

\bibitem[Gen17]{Gendulphe}
Matthieu Gendulphe.
\newblock {W}hat's wrong with the growth of simple closed geodesics on
  nonorientable hyperbolic surfaces.
\newblock {\em \href{https://arxiv.org/abs/1706.08798}{arXiv:1706.08798}},
  2017.

\bibitem[Glo17]{Glorious}
Olivier Glorieux.
\newblock {C}ritical exponent for geodesic currents.
\newblock {\em \href{https://arxiv.org/abs/1704.06541}{arXiv:1704.06541}},
  2017.

\bibitem[HP97]{HerPa}
Sa'ar Hersonsky and Fr\'ed\'eric Paulin.
\newblock On the rigidity of discrete isometry groups of negatively curved
  spaces.
\newblock {\em Comment. Math. Helv.}, 72(3):349--388, 1997.

\bibitem[Iva92]{Iva}
Nikolai~V. Ivanov.
\newblock {\em Subgroups of {T}eichm\"uller modular groups}, volume 115 of {\em
  Translations of Mathematical Monographs}.
\newblock American Mathematical Society, Providence, RI, 1992.
\newblock Translated from the Russian by E. J. F. Primrose and revised by the
  author.

\bibitem[KB02]{col:KB02}
Ilya Kapovich and Nadia Benakli.
\newblock Boundaries of hyperbolic groups.
\newblock In {\em Combinatorial and geometric group theory ({N}ew {Y}ork,
  2000/{H}oboken, {NJ}, 2001)}, volume 296 of {\em Contemp. Math.}, pages
  39--93. Amer. Math. Soc., Providence, RI, 2002.

\bibitem[KN13]{KN}
Ilya Kapovich and Tatiana Nagnibeda.
\newblock Subset currents on free groups.
\newblock {\em Geometriae Dedicata}, 166(1):307--348, 2013.

\bibitem[LM08]{LM}
Elon Lindenstrauss and Maryam Mirzakhani.
\newblock Ergodic theory of the space of measured laminations.
\newblock {\em Int. Math. Res. Not. IMRN}, (4):Art. ID rnm126, 49, 2008.

\bibitem[Mag17]{Magee}
Michael Magee.
\newblock {C}ounting one sided simple closed geodesics on fuchsian thrice
  punctured projective planes.
\newblock {\em \href{https://arxiv.org/abs/1705.09377}{arXiv:1705.09377}},
  2017.

\bibitem[Mas85]{MasurPAMS}
Howard Masur.
\newblock Ergodic actions of the mapping class group.
\newblock {\em Proceedings of the American Mathematical Society},
  94(3):455--459, 1985.

\bibitem[Mir08]{Mir1}
Maryam Mirzakhani.
\newblock Growth of the number of simple closed geodesics on hyperbolic
  surfaces.
\newblock {\em Ann. of Math. (2)}, 168(1):97--125, 2008.

\bibitem[Mir16]{Mir2}
Maryam Mirzakhani.
\newblock Counting mapping class group orbits on hyperbolic surfaces.
\newblock {\em \href{https://arxiv.org/abs/1601.03342}{arXiv:1601.03342}},
  2016.

\bibitem[MZ16]{MZ}
Giuseppe Martone and Tengren Zhang.
\newblock Positively ratioed representations.
\newblock {\em arXiv:1609.01245}, 2016.

\bibitem[Ota90]{Otal}
Jean-Pierre Otal.
\newblock Le spectre marqu\'e des longueurs des surfaces \`a courbure
  n\'egative.
\newblock {\em Ann. of Math. (2)}, 131(1):151--162, 1990.

\bibitem[PH92]{PH}
Robert~C Penner and John~L Harer.
\newblock Combinatorics of train tracks.
\newblock {\em Annals of Mathematics Studies}, (125):3--114, 1992.

\bibitem[RS17]{RS}
Kasra Rafi and Juan Souto.
\newblock Geodesics currents and counting problems.
\newblock {\em \href{https://arxiv.org/abs/1709.06834}{arXiv:1709.06834}},
  2017.

\bibitem[Sas17]{SaS}
Dounnu Sasaki.
\newblock Subset currents on surfaces.
\newblock {\em arXiv:1703.05739}, 2017.

\bibitem[Thu80]{ThurstonNotes}
William~P. Thurston.
\newblock {T}he geometry and topology of 3-manifolds.
\newblock {\em Unpublished notes}, 1980.

\bibitem[Thu88]{Th}
William~P. Thurston.
\newblock On the geometry and dynamics of diffeomorphisms of surfaces.
\newblock {\em Bull. Amer. Math. Soc. (N.S.)}, 19(2):417--431, 1988.

\bibitem[Uya15]{UyaNSD}
Caglar Uyanik.
\newblock Generalized north-south dynamics on the space of geodesic currents.
\newblock {\em Geom. Dedicata}, 177:129--148, 2015.

\end{thebibliography}
\bibliographystyle{alpha}

\end{document}